\documentclass[11pt]{amsart}

\usepackage[english]{babel}
\usepackage[ansinew]{inputenc}
\usepackage{amsmath, amsthm, amssymb, graphicx, mathrsfs}
\usepackage{color}
\usepackage{hyperref}
\usepackage{pinlabel}
\usepackage{tikz, pbox}
\usepackage{tkz-euclide}
\usetikzlibrary{decorations.markings,arrows}
\usetikzlibrary{patterns}

\usepackage[headings]{fullpage}

\usepackage{amsmath,amsthm,amssymb}

\newtheorem{thm}{Theorem}[section]
\newtheorem{lem}[thm]{Lemma}
\newtheorem{cor}[thm]{Corollary}

\newtheorem{prop}[thm]{Proposition}

\newtheorem{note}[thm]{Note}
\theoremstyle{definition}
\newtheorem{defn}[thm]{Definition}

\newtheorem{example}[thm]{Example}
\newtheorem{prop-defn}[thm]{Proposition-Definition}

\newtheorem*{claim*}{Claim}
\newtheorem*{ack*}{Acknowledgements}
\newtheorem*{ex*}{Example}

\renewcommand{\vec}[1]{\mathbf{#1}}

\newtheoremstyle{named}{}{}{\itshape}{}{\bfseries}{.}{.5em}{#3}
\theoremstyle{named}
\newtheorem*{namedtheorem}{Corollary}

\title{A Transversal for horocycle flow on $\mathcal{H}(2)$}
\author{Grace Work}
\address{\tt Department of Mathematics, University of Illinois at
 Urbana-Champaign, 1409 West Green Street, Urbana, IL 61801, USA
\newline http://www.math.uiuc.edu/\~{}work2} \email{\tt work2@illinois.edu}
\thanks{\today}

\begin{document}

\thanks{The author acknowledges support from U.S. National Science Foundation grants DMS 1107452, 1107263, 1107367 ``RNMS: GEometric structures And Representation varieties" (the GEAR Network).}

\maketitle
	
\begin{abstract} Using zippered rectangle coordinates we parametrize a Poincar\'e section for horocycle flow on the space of genus 2 translation surfaces with one singular cone point of angle $6\pi$. In addition, we bound the return time under horocycle flow to this Poincar\'e section by examining a subset of surfaces where a certain sum of parameters is large.
\end{abstract}

\section{Introduction}

\subsection{A transversal to horocycle flow on moduli space}

Let $\mathcal{H}(\alpha)$ be a stratum of the moduli space of genus $g$ translation surfaces, where a genus $g$ translation surface is a pair $(S, \omega)$ where $S$ is a genus $g$ Riemann surface and $\omega$ a holomorphic one-form. Here, $\alpha$ is an integer partition of $2g-2$ which gives the orders of the zeros of $\omega$. A translation surface can be realized as a collection of polygons in the plane with parallel sides identified by translations, and a zero of order $k$ for $\omega$ corresponds to a cone point of the flat metric induced by the polygons with angle $2\pi(k+1)$. There is an $SL(2, \mathbb{R})$ action on this stratum, and, if the stabilizer of a given surface under this action is a lattice, we call the surface a \emph{lattice surface}, or \emph{Veech surface}; generically the stabilizer is trivial.

A \emph{saddle connection} on a translation surface is a straight line trajectory, $\gamma$, connecting two, not necessarily distinct, cone points, with no cone points in its interior. To each saddle connection we can associate a holonomy vector $\mathbf{v}_\gamma$ that records how far $\gamma$ travels in the horizontal and the vertical direction. The angles of holonomy vectors associated to saddle connections of length at most $R$ have been proven, by Veech in the lattice surface case [12] and by Vorobets for almost every surface [13], to equidistribute (as $R \rightarrow \infty$) with respect to Lebesgue measure on the circle. If we consider the set of holonomy vectors with slopes in between 0 and 1, and horizontal component at most $R$, then these also equidistribute with respect to Lebesgue measure on $(0, 1)$ as $R \rightarrow \infty$, by work of Athreya [1].

The finer statistic of the distribution of the gaps between consecutive elements in the sequence of slopes of saddle connections has been studied in the case of lattice surfaces. First by Athreya and Cheung who examined the case of the torus \cite{AC13}, then by Athreya, Chaika, and Leli\`evre in the case of the double pentagon \cite{ACL}, and finally by the author and Uyanik in the case of the octagon and general Veech surfaces \cite{UW}. These results follow a strategy of proof outlined by Athreya involving translating the question of the gap distribution into a dynamical question of return times of \emph{horocycle flow} to a specific \emph{Poincar\'e section}, defined to be the set of all surfaces with a horizontal saddle connection of length $\leq 1$ \cite{Ath13}. Computing the distribution of the return time to this Poincar\'e section yields the gap distribution for slopes. In joint work with Uyanik we describe this Poincar\'e section for all lattice surfaces \cite{UW}.

Athreya and Chaika showed for almost every surface, the angle gap distribution exists (and is constant on this full measure set), and that it has support at zero \cite{AC12}. Moreover, Athreya showed that the analagous statement for slope gap distributions holds by studying transversals to the horocycle flow \cite{Ath13}. We show

\begin{thm}
There is a Poincar\'e section for horocycle flow on the moduli space $\mathcal{H}_1(\alpha)$, the space of area one surfaces in $\mathcal{H}(\alpha)$, parametrized by a union of $n$ polytopes, where $n$ is the cardinality of the associated Rauzy class.
\end{thm}

We will provide an explicit description of this Poincar\'e section the case of $\mathcal{H}_1(2)$ and give bounds for the return time by considering a subset of these surfaces. These are the first results leading towards an explicit computation for distribution of the gaps between slopes of holonomy vectors associated to saddle connections the case of a generic surface.

\subsection{Translation surfaces}
We now review the necessary background on translation surfaces and refer the reader to \cite{HS, Masur, MT, Zorich} for a more detailed introduction. A \emph{translation surface} is a pair $(S, \omega)$ where $S$ is a Riemann surface and $\omega$ is a holomorphic one-form. Equivalently, one can form a translation surface by identifying parallel sides of a collection of polygons embedded in the plane by translations.

Every translation surface has three pieces of associated topological data: the genus, the set of zeros, and the multiplicity of its singularities. This data can be represented by a vector $\alpha = (\alpha_1, \ldots, \alpha_k)$ where $\alpha_i$ is the order of the $i$th zero and corresponds to a point of total angle $2\pi(\alpha_i + 1)$. in addition, the $\alpha_i$ satisfy the following equality
\[\sum_{i=1}^k \alpha_i = 2g - 2\]

Two translation surfaces are distinct if they differ by a nontrivial rotation, that is $(S, \omega)$ and $(S, e^{i\theta}\omega)$ are typically not equivalent if $\theta \neq 0$. We denote by $\mathcal{H}(\alpha)$ the \emph{stratum} of translation surfaces with topological data $\alpha$ and by $\mathcal{H}_1(\alpha)$ the set of translation surfaces in $\mathcal{H}(\alpha)$ with area 1.

There is a natural $SL(2, \mathbb{R})$ action on the stratum, given a matrix $A \in SL(2, \mathbb{R})$ and a translation surface defined by the polygons $\{P_1, \ldots, P_k\}$ we obtain a new translation surface defined by $\{AP_1, \ldots, AP_k\}$. The stabilizer of a given translation surface, $(S, \omega)$, under this action is called the \emph{Veech group}, and is denoted by $SL(S, \omega)$. If this stabilizer is a lattice, that is if $SL(2, \mathbb{R})/SL(S, \omega)$ has finite volume, the surface is called a \emph{lattice surface}, or a \emph{Veech surface}. Generically this stabilizer is trivial. The regular octagon with parallel sides identified is an example of a lattice surface in $\mathcal{H}(2)$ and its Veech group is isomorphic to $\triangle(4, \infty, \infty)$.

A straight line trajectory on these surfaces connecting two, not necessarily distinct, cone points and containing no cone points in its interior, is called a \emph{saddle connection}. To each saddle connection, $\gamma$, we can associate a \emph{holonomy vector}, $\vec{v}_\gamma = \int_\gamma\omega \in \mathbb{C}$. The set of holonomy vectors is a discrete subset of $\mathbb{R}^2$

\subsection{Gap distributions}
Studying the distribution of the gaps between elements of an equidistributed sequence provides a finer test for randomness. Were the sequence truly random we would expect an exponential distribution, as is the case with independent, identically distributed, uniform on $[0,1]$ random variables. One of the first uses of this test was in 1970, when R. R. Hall studied the gap distribution for the Farey sequence \cite{Hall70}. Later Elkies and McMullen studied the distribution of gaps in the sequence $\sqrt{n}\mod 1$ using techniques from homogeneous dynamics, and dubbed these non-exponential distributions ``exotic'' \cite{EM}.

In 2013 Athreya and Cheung studied slopes of straight line trajectories on a square with parallel sides identified, this question is equivalent to examining the Farey sequence. Their proof used the ergodic theory of the horocycle flow, where \emph{horocycle flow} is defined to be the action of the matrix
\[h_s = \begin{bmatrix} 1 & 0 \\ -s & 1\end{bmatrix} \in SL(2, \mathbb{R}).\]
They were able to obtain Hall's distribution by computing the distribution of the return time of horocycle flow to a specific Poincar\'e section  \cite{AC13}. In 2014 Athreya, Chaika, and Leli\`{e}vre applied the same strategy to compute the distribution of gaps between slopes of saddle connections on the golden $L$, a translation surface in $\mathcal{H}(2)$ with Veech group $\triangle(2, 5, \infty)$ \cite{ACL}. This led to considering surfaces whose Veech group had more than one cusp. The first example of this type was the octagon, a Veech surface in $\mathcal{H}(2)$ with Veech group $\triangle(4, \infty, \infty)$, computed by the author and Uyanik in 2015, where it was also proven that the Poincar\'{e} section for a general Veech surface whose Veech group has $n$ cusps can be parametrized by a union of $n$ triangles and the gap distribution will be piecewise real analytic \cite{UW}.

\begin{ack*}
We would like to thank Jayadev Athreya for proposing this problem, and Alex Wright for suggesting zippered rectangle coordinates as a natural parametrization for the transversal, as well as Vincent Delecroix, Samuel Leli\`evre, and Ronen Mukamel for useful discussions and help with code, and MSRI for its hospitality during the Spring 2015 semester. 
\end{ack*}

\section{Constructing the Poincar\'e Section}

A key step in the computation of the gap distribution using the ergodicity of the horocycle flow is to the find a good parametrization of the Poincar\'{e} section, that is, the set of surfaces in $\mathcal{H}_1(\alpha)$ with a short, length $\leq 1$, horizontal saddle connection. To do this we must first construct coordinates on this space.

\subsection{Zippered rectangles}
We begin by reviewing the zippered rectangle construction by Veech, \cite{Vee82}. Every surface in $\mathcal{H}_1(\alpha)$ has a zippered rectangle representation by constructing a suspension over an interval exchange transformation (IET), built by finding the first return map of vertical flow to a given transversal, in our case the longest horizontal saddle connection of length $\leq 1$. We show a construction in Figure \ref{ZR}, for a surface $S \in \mathcal{H}_1(2)$, here the transversal was chosen to be the saddle connection from $a_0$ to $a_4$. Associated to each zippered rectangle is a vector $\lambda$, indicating the lengths of the intervals, and a permutation $\pi$ corresponding to the interval exchange transformation. In addition there are the vectors $h$, the heights of the rectangles, and $a$, the altitudes of the cone points. In order for the zippered rectangle to be valid the vectors $a$ and $h$ must satisfy the following: \cite{Z96}
\begin{align}
h_i - a_i &= h_{\sigma(i) + 1} - a_{\sigma(i)} & (0 \leq i \leq m)\\
h_i, a_i &\geq 0 &(1 \leq i \leq m)\\
h_m \geq a_m &\geq -h_{\pi^{-1}m}&\\
h_{\pi^{-1}m+1} & \geq a_{\pi^{-1}m} &\\
\min(h_i, h_{i+1}) & \geq a_i & (0 < i < m, i \neq \pi^{-1}m)
\end{align}

where $m$ is the number of rectangles and

\[ \sigma(j) = \left\{\begin{array}{l l} \pi^{-1}(1) -1 & j = 0\\ m & j = \pi^{-1}(m) \\ \pi^{-1}(\pi(j)+1)-1 & \text{otherwise}\end{array}\right..\]

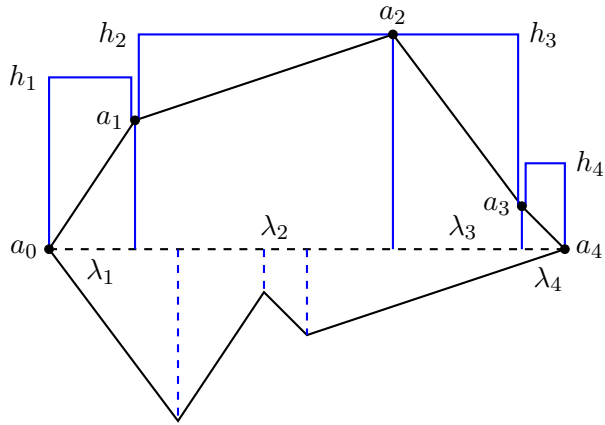
\begin{figure}[ht!]\label{ZR}
\begin{center}
\begin{tikzpicture}
\coordinate (A0) at (0,0);
\node at (A0) [left = .1mm of A0] {$a_0$};
\coordinate (H11) at (0,4/1.75);
\node at (H11) [left = .1mm of H11] {$h_1$};
\coordinate (H12) at (2/1.75-1/20,4/1.75);
\coordinate (A11) at (2/1.75-1/20,3/1.75);
\coordinate (A1) at (2/1.75,3/1.75);
\node at (A1) [left = .1mm of A1] {$a_1$};
\coordinate (A12) at (2/1.75+1/20,3/1.75);
\coordinate (H2) at (2/1.75+1/20,5/1.75);
\node at (H2) [left = .1mm of H2] {$h_2$};
\coordinate (A2) at (8/1.75,5/1.75);
\node at (A2) [above = .1mm of A2] {$a_2$};
\coordinate (H3) at (11/1.75-1/20,5/1.75);
\node at (H3) [right = .1mm of H3] {$h_3$};
\coordinate (A31) at (11/1.75-1/20,1/1.75);
\coordinate (A3) at (11/1.75,1/1.75);
\node at (A3) [left = .1mm of A3] {$a_3$};
\coordinate (A32) at (11/1.75+1/20,1/1.75);
\coordinate (H41) at (11/1.75+1/20,2/1.75);
\coordinate (H42) at (12/1.75,2/1.75);
\node at (H42) [right = .1mm of H42] {$h_4$};
\coordinate (A4) at (12/1.75,0);
\node at (A4) [right = .1mm of A4] {$a_4$};
\coordinate (L1) at (2/1.75,0);
\coordinate (L2) at (8/1.75,0);
\coordinate (L3) at (11/1.75,0);
\coordinate (B1) at (3/1.75,0);
\coordinate (B2) at (5/1.75,0);
\coordinate (B3) at (6/1.75,0);
\coordinate (D1) at (3/1.75,-4/1.75);
\coordinate (D2) at (5/1.75,-1/1.75);
\coordinate (D3) at (6/1.75,-2/1.75);

\node at (11/16,0) [below] {$\lambda_1$};
\node at (3,0) [above] {$\lambda_2$};
\node at (4+3/2,0) [above] {$\lambda_3$};
\node at (11/1.75+.35,-.1) [below] {$\lambda_4$};

\draw[thick] (A0) -- (A1) -- (A2) -- (A3) -- (A4) -- (D3) -- (D2) -- (D1) -- cycle;

\draw[thick, dashed] (A0) -- (A4);

\foreach \i in {1, 2, 3}
{
\draw[thick,dashed,blue] (B\i) -- (D\i);
\draw[thick,blue] (L\i) -- (A\i);
}

\draw[thick,blue] (A0) -- (H11) -- (H12) -- (A11);
\draw[thick,blue] (A12) -- (H2) -- (A2) -- (H3) -- (A31);
\draw[thick,blue] (A32) -- (H41) -- (H42) -- (A4);

\foreach \L in {A0, A1, A2, A3, A4}
{
\draw[fill=black] (\L) circle (0.15em);
}
\end{tikzpicture}
\end{center}
\caption{Constructing a zippered rectangle with permutation $\pi = (3142)$ and coordinates $a=(a_0, a_1, a_2, a_3, a_4)$, $h =(h_1, h_2, h_3, h_4)$, and $\lambda=(\lambda_1,\lambda_2,\lambda_3,\lambda_4)$, from a polygon representation of a surface.}
\end{figure}

The genus and number of singularities of a translation surface depend only on the Rauzy class of the permutation, thus there are only $n$ possibilities for the permutation $\pi$, where $n$ is the cardinality of the associated Rauzy class. For example, in the stratum $\mathcal{H}_1(2)$, there are 7 possibilities for $\pi$, those in the Rauzy class of $(4321)$.

\subsection{Constructing coordinates}
Consider a translation surface $S$ in $\mathcal{H}_1(\alpha)$, $\alpha = (\alpha_1, \ldots, \alpha_k)$. If, as in our case, we choose the transversal, $X$, to be a horizontal saddle connection, then the IET $T : X\rightarrow X$ induced by the first return of the vertical flow on $S$ to $X$ has the minimal possible number, $m = 2g + k -1$, of subintervals under exchange. Since the endpoints of $X$ coincide with cone points, an interior point $x \in X$ will be a point of of discontinuity for $T$ only if it intersects a cone point in forward time under the vertical trajectory before returning to $X$. For a cone point having angle $2\pi(\alpha_i+1)$, there are $\alpha_i+1$ vertical trajectories that will meet it, and thus these correspond to $\alpha_i+1$ points of discontinuities on $X$. We will have $\sum_{i=1}^k(\alpha_i + 1)$ points of discontinuity on $X$ and therefore $m = \sum_{i=1}^n(\alpha_i + 1) + 1$ subintervals. Using the Gauss-Bonnet formula, $\sum_{i=1}^n\alpha_i= 2g-2$ we get that $m = 2g + n -1$ \cite{Zorich}. Thus, for $M \in \mathcal{H}_1(2)$ and $X$ a horizontal saddle connection of length $\leq 1$, we have that $X$ has $m= 2\cdot 2 + 1 - 1 = 4$ subintervals under the exchange.

The coordinates from Zorich \cite{Z96} give $(a, h, \lambda)$ where $a \in \mathbb{R}^{m+1}$ and $h, \lambda \in \mathbb{R}^m$. Let $d = \text{dim}_{\mathbb{R}}\mathcal{H}(\alpha) = 2m$. Since we are working in the space $\mathcal{H}_1(\alpha)$ and have restricted $X$ to be a horizontal saddle connection, we should be able to find $d-2$ independent coordinates. The coordinates $h$ and $a$ satisfy (1) -- (5) for $0 \leq i \leq m$, and we use the dummy components $h_0 = h_{m+1} = a_0 = 0$. We are thus able to rewrite each element in $h$ in terms of $a$.

Restricting to area one surfaces also gives the equation
\[\sum\lambda_ih_i = 1\]
and requiring the surface to have a horizontal saddle connection of length $< 1$, gives the equality $a_m = 0$ and the inequality
\[\sum\lambda_i < 1\]
This leaves $d-2$ coordinates: $(a_2, \ldots, a_{m-1}; \lambda)$.

\begin{example}
Let $\pi = (4321)$. Then $\sigma(j) = j+3 \mod 5$ and we get the following equalities:
\begin{alignat*}{5}
h_0 - a_0 & =  h_4 - a_3  =  0  \hspace{10pt} &\Longrightarrow&  \hspace{10pt} h_4 = a_3\\
h_1 - a_1 & =  h_5 - a_4  =  0  \hspace{10pt}&\Longrightarrow& \hspace{10pt}  h_1 = a_1\\
h_2 - a_2 & =  h_1 - a_0  =  h_1   \hspace{10pt}&\Longrightarrow& \hspace{10pt}  h_2 = a_1 + a_2\\
h_3 - a_3 & =  h_2 - a_1       \hspace{10pt}&\Longrightarrow& \hspace{10pt}  h_3 = a_3 + a_2\\
h_4 - a_4 & =  h_3 - a_2  =  h_4  &&
\end{alignat*}
Now we have that $\sum_{i=1}^m h_i\lambda_i = 1$ so we can write
\[h_1 = \frac{1 - \sum_{i=2}^m h_i\lambda_i}{\lambda_1}\]
Substituting in the $a_i$'s using the above equalities and solving for $a_1$, gives
\[a_1 = \frac{1-a_2\lambda_2 - (a_3+a_2)\lambda_3 - a_3\lambda_4}{\lambda_1 + \lambda_2}\]
\end{example}

\begin{namedtheorem}[Theorem 1.1]
The Poincar\'e section for horocycle flow on the moduli space $\mathcal{H}_1(\alpha)$ can be parametrized by a union of $n$ polytopes, where $n$ is the cardinality of the associated Rauzy class.
\end{namedtheorem}

\begin{proof}
Let $M \in \mathcal{H}_1(\alpha)$ have a horizontal saddle connection of length $\leq 1$. If there is more than one such saddle connection, we will choose the longest. Constructing the suspension over the IET built by finding the first return map of vertical flow to a given transversal will produce a zippered rectangle described by the given inequalities and with coordinates lying in one of the $n$ polytopes corresponding to the associated permutation $\pi$.

Let $Z$ be an element in the Poincar\'e section, since $\pi$ is in the associated Rauzy class, it will correspond to a translation surface, $M \in \mathcal{H}_1(\alpha)$. Since $a_m = 0$ and $a_i > 0$ for $i \neq m$, the transversal will be a saddle connection in $M$, and with the restriction $\sum_{i}\lambda_i \leq 1$, it will have length $\leq 1$.

The return time function is calculated by finding the saddle connection with horizontal component $\leq 1$ and smallest slope. The slope is a rational function in these coordinates, and thus so is the return time function.
\end{proof}

\begin{cor}
The Poincar\'e section, $\Omega$, for horocycle flow on $\mathcal{H}(2)$ is parametrized by a union of 7 polytopes described by linear inequalities in the coordinates $(a_2, a_3; \lambda_1, \lambda_2, \lambda_3, \lambda_4)$ and the return time function is piecewise rational in these coordinates.
\end{cor}

In section \ref{PSRTB} we give an explicit description of the polytopes and bounds.

\section{Computing the Return Time Function}
\subsection{Bounding return times}
We begin by fixing notation that will be used in the proof of the bounds. Denote the cone points as $x_0, \ldots, x_4$ where $x_i$ has height $a_i$, and the rectangles by $R_1, \ldots, R_4$. In this way $x_i$ lies on the right-hand side of $R_i$, for $i \in \{1, 2, 3, 4\}$ and on the left-hand side of $R_{i+1}$ for $i \in \{0, 1, 2, 3\}$. A gluing identification $R_j$ glued to $R_{\sigma(j)+1}$ will be denoted by $R_jR_{\sigma(j)+1}$. We will denote saddle connections by $(R_{i_1}, \ldots, R_{i_n})$ which indicates a saddle connection starting at $x_{i_1-1}$ and ending at $x_{i_n}$, and passing through the rectangles $R_{i_1}, \ldots, R_{i_n}$.

To bound the return times, we assume $\sum \lambda_i + \min(\lambda_i) > 1$. We first observe the following
\begin{note}\label{RED}{\ }
\begin{enumerate}
\item $(R_1)$ will always be a saddle connection, since $h_1 \geq a_1$, and thus its slope $\frac{a_1}{\lambda_1}$ will be a universal upper bound for the return time.

\item If $a_i = h_{i+1}$, there will be no saddle connections beginning at $x_i$. In  particular, there will be no saddle connections beginning at $x_{\sigma(0)}$ due to the equality $h_0 - a_0 = h_{\sigma(0) + 1} - a_{\sigma(0)}$.

\item No saddle connections end at $x_4$.

\item If $a_i \geq a_{i+1}$, a candidate vector will never start by passing to the adjacent rectangle.

\item If $a_i < a_{i+1}$, a saddle connection beginning at $x_i$ and passing first through a gluing identification will never have smallest slope, since the vector connecting $x_i$ to $x_{i+1}$ will always exist and have smaller slope. In particular, a saddle connection beginning at $x_0$ will never start by passing through a gluing identification.

\item If $\sigma(i) = 4$, all saddle connections must pass from $R_i$ to $R_{i+1}$.
\end{enumerate}
\end{note}

\begin{lem}
If a saddle connection starts and ends at the same point, there will always be a saddle connection with the same slope that starts and ends at different points.
\end{lem}
\label{diffAlt}
\begin{proof}
Let $\gamma$ be a saddle connection starting and ending at $x_i$ with positive slope. Then $\gamma$ must pass through a gluing identification $R_jR_{\sigma(j)+1}$. Let $y_j$ be the height at which it exits the left side of rectangle $R_j$ and $y_{\sigma(j) + 1}$ be the height at which it enters the right side of $R_{\sigma(j)+1}$. Then $y_j - a_j = y_{\sigma(j)+1} - a_{\sigma(j)}$, subtracting this value from the height of every point on $\gamma$ yields a saddle connection $\gamma'$ with the same slope starting at $x_{\sigma(j)}$ and ending at $x_j$. If $\gamma$ passes through multiple gluing identifications, $R_{j_i}R_{\sigma(j_i)+1}$, we pick the identification associated to $\min_i\{y_{j_i}- a_{j_i}\}$. The saddle connection remains valid as it only cycles the sequence of rectangles.
\end{proof}
\begin{example}
Consider the following zippered rectangle picture where $\pi = (3142)$ and $\gamma$ is a saddle connection connecting $x_1$ to itself. $\gamma$ passes through the gluing identification $R_3R_1$ and so there is a new saddle connection $\gamma'$ with the same slope that starts at $x_0$ and ends at $x_3$.
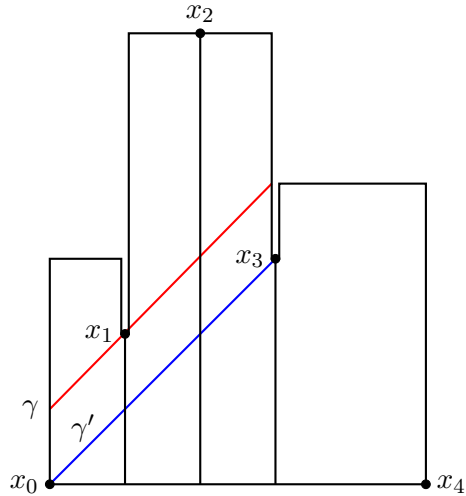
\begin{figure}[ht!]
\begin{center}
\begin{tikzpicture}
\coordinate (A0) at (0,0);
\node at (A0) [left = .1mm of A0] {$x_0$};
\coordinate (H11) at (0,3);
\coordinate (H12) at (1-1/20,3);
\coordinate (A11) at (1-1/20,2);
\coordinate (A1) at (1,2);
\node at (A1) [left = .1mm of A1] {$x_1$};
\coordinate (A12) at (1+1/20,2);
\coordinate (H2) at (1+1/20,6);
\coordinate (A2) at (2,6);
\node at (A2) [above = .1mm of A2] {$x_2$};
\coordinate (H3) at (3-1/20,6);
\coordinate (A31) at (3-1/20,3);
\coordinate (A3) at (3,3);
\node at (A3) [left = .1mm of A3] {$x_3$};
\coordinate (A32) at (3+1/20,3);
\coordinate (H41) at (3+1/20,4);
\coordinate (H42) at (5,4);
\coordinate (A4) at (5,0);
\node at (A4) [right = .1mm of A4] {$x_4$};

\coordinate (L1) at (1,0);
\coordinate (L2) at (2,0);
\coordinate (L3) at (3,0);

\coordinate (G1) at (0,1);
\node at (G1) [left = .1mm of G1] {$\gamma$};
\coordinate (G2) at (3-1/20,4);

\coordinate (GG) at (.75,.75);
\node at (GG) [left = .1mm of GG] {$\gamma'$};

\draw[thick, red] (G1) -- (G2);
\draw[thick, blue] (A0) -- (A3);

\draw[thick] (A0) -- (H11) -- (H12) -- (A11) -- (A12) -- (H2) -- (A2) -- (H3) -- (A31) -- (A32) -- (H41) -- (H42) -- (A4) -- cycle;
\draw[thick] (L1) -- (A1);
\draw[thick] (L2) -- (A2);
\draw[thick] (L3) -- (A3);

\foreach \A in {A0, A1, A2, A3, A4}
{
\draw[fill=black] (\A) circle (0.15em);
}
\end{tikzpicture}
\end{center}
\caption{The saddle connection $\gamma'$ with the same slope as $\gamma$ but starting and ending at different altitudes}
\label{disk}
\end{figure}

\end{example}

\begin{lem}\label{comp}{\ }
\begin{enumerate}
\item Consider two distinct saddle connections $\gamma$ and $\gamma'$ with the same initial sequence $(R_{i_1},\ldots, R_{i_k}, \ldots)$. If $R_{i_{k+1}} = R_{i_k +1}$ in $\gamma'$, while in $\gamma$ it does not exist or is $R_{\sigma(i_k)+1}$, then $\gamma'$ has smaller slope.
\item Consider two distinct saddle connections $\gamma$ and $\gamma'$ with the same terminal sequence $(\ldots, R_{i_k}, \ldots, R_{i_n})$. If $R_{i_{k-1}} = R_{\sigma^{-1}(i_k-1)}$ or does not exist in $\gamma'$ while in $\gamma$ it does not exist or is $R_{i_k -1}$, then $\gamma'$ has smaller slope.
\end{enumerate}
\end{lem}

\begin{proof}{\ }
\begin{enumerate}
\item Let $\ell = \sum_{j=1}^k \lambda_{i_j}$ be the horizontal distance traveled by both $\gamma$ and $\gamma'$ over the initial sequence. Let $y$ and $y'$ be the heights at which $\gamma$ and $\gamma'$, respectively, intersect the left boundary of $R_{i_k}$. Since $R_{i_{k+1}} = R_{i_k +1}$ in $\gamma'$, this implies that $y' < a_{i_k}$, on the other hand $y \geq a_{i_k}$, since $R_{i_{k+1}}$ in $\gamma$ either does not exist, implying $\gamma$ terminates at $x_{i_k}$, or is $R_{\sigma(i_k)+1}$, implying it must pass above $a_{i_k}$. Thus $\frac{y'}{\ell} < \frac{y}{\ell}$ and so $\gamma'$ has smaller slope.
\item In a similar manner, let $\ell = \sum_{j=k}^n \lambda_{i_j}$ be the horizontal distance traveled by both $\gamma$ and $\gamma'$ over the terminal sequence. Let $y$ and $y'$ be the heights at which $\gamma$ and $\gamma'$, respectively, intersect the right boundary of $R_{i_k}$. In $\gamma'$, $R_{i_{k-1}} = R_{\sigma^{-1}(i_k-1)}$ or it does not exist, so we have $y' \geq a_{{i_k}-1}$, in $\gamma$, $R_{i_{k-1}} = R_{i_k -1}$ or it does not exists, so we have $y \leq a_{{i_k}-1}$. We also know $y' \neq y$ since $\gamma$ and $\gamma'$ are distinct. Thus $a_i - y' < a_i - y$ and so $\gamma'$ again has smaller slope.
\end{enumerate}
\end{proof}

\begin{lem} 
If $\sum \lambda_i + \min(\lambda_i) > 1$, then the saddle connection with smallest slope will not pass through the top of a rectangle
\end{lem}
\begin{proof} A saddle connection cannot start at $x_0$ and go through the top of $R_i$, $1 \leq i \leq 3$, because if it follows a valid path, this would imply $(R_1, \ldots, R_i)$ exists and has smaller slope. 

\textbf{Claim.} A saddle connection that passes the transversal cannot immediately terminate at any $x_i$, therefore once it passes through the transversal it must continue through the right side of $R_4$ to $R_{\sigma(4)+1}$.

\textbf{Proof of claim.} No saddle connections end at $x_0$ or $x_4$. If the saddle connection terminates at $x_1$, $(R_1)$ always exist and has smaller slope. If the saddle connection terminates at $x_2$, either $(R_1, R_2)$ exists and has smaller slope, if $a_1 \geq a_2$, or $(R_2)$ exists and has smaller slope, if $a_1 < a_2$. If the saddle connection terminates at $x_3$, then either $(R_1, R_2, R_3)$ exists and has smaller slope, if $a_1 \geq a_2 \geq a_3$, or $(R_2, R_3)$ exists and has smaller slope, if $a_1 < a_2 \geq a_3$, or $(R_3)$ exists and has smaller slope, if $a_2 < a_3$.

The claim also implies that a saddle connection can never pass through the top of $R_{\pi^{-1}(1)}$, in this case it would have minimum length $\lambda_1 + \lambda_2 + \lambda_3 + \lambda_4 + \lambda_{\sigma(4)}$, and is thus too long.

To complete the proof of the lemma we examine six cases based on the relationships between the $a_i$.\\

\noindent\textbf{Case 1:} ($a_1 < a_2 < a_3$). The only possible permutations are $\pi = (4321), (4132), (4213)$, and in all of these cases no saddle connections can begin at $x_{\sigma(0)} = x_3$, by Note \ref{RED} (2). Thus we only need to consider $x_1$ and $x_2$. No saddle connection can start at $x_i$ and pass through the top of $R_{i+1}$, for $i = 1, 2$, since $(R_{i+1})$ exists and will have smaller slope. Since $\pi^{-1}(1) = 4$ for all the possible permutations, no saddle connection can pass though the top of $R_4$.

If the saddle connection starts at $x_1$, it cannot pass through the top of $R_3$ as $(R_3)$ exists and has smaller slope. It cannot pass through the top of $R_1$ as this would require a minimum length of $\lambda_2 + \lambda_3 + \lambda_4 + \lambda_1 + \lambda_{\sigma(4)} > 1$. If the saddle connection starts at $x_2$, it cannot pass through the top of $R_1$ or $R_2$ as this would require a minimum length of $\lambda_3+\lambda_4+\lambda_1+\lambda_2+\lambda_3 > 1$.

\noindent\textbf{Case 2:} ($a_1 \geq a_2 \geq a_3$). Here the possible permutations are $(2431)$ and $(4321)$. If we consider the surface that the zippered rectangles came from, we see that the horizontal saddle connection is in the middle separating the surface into two pieces with 3 vertices above the transversal, $x_1, x_2$, and $x_3$, and 3 vertices below. $x_i', x_2'$, and $x_3'$. Consider the 4 saddle connections from $x_0$ to $x_1$, $x_2$, and $x_3$ and from $x_1'$ and $x_2'$ to $x_4$, each of these exist since the $a_i$ are in decreasing order. Any saddle connection that passes through the top of a rectangle will either be too long or have slope greater than the slope of any of these 4 saddle connections.\\

\noindent\textbf{Case 3:} ($a_1 \geq a_3 > a_2$).  A saddle connection cannot start at $x_2$ and immediately pass through the top of $R_3$, since $(R_3)$ always exists and will have smaller slope. 

\textbf{Case 3.a:} $\pi = (4321)$. No saddle connection can pass through the top of $R_4$. A saddle connection starting at $x_2$ cannot pass through the top of $R_1$ or $R_2$, as this would require a minimum length of $\lambda_3 + \lambda_4 + \lambda_3 + \lambda_2 + \lambda_1$. Since no saddle connections can begin at $x_3$ in this permutation, it only remains to check $x_1$. A saddle connection cannot begin at $x_1$ and pass trough the top of $R_1$ or $R_2$ as it will be beaten by the saddle connection $(R_1, R_2)$. In addition, it cannot pass through the top of $R_3$ as this implies it would have had to enter $R_3$ from $R_2$ and thus $(R_3)$ will have smaller slope. 

\textbf{Case 3.b:} $\pi = (2431), (2413)$. No saddle connection can pass through the top of $R_2$. With both of these permutations no saddle connections can begin at $x_1$. Thus it remains to check $x_2$ and $x_3$. No saddle connection can start at $x_0$ and pass through the top of $R_4$ as this would require a minimum length of at least $\lambda_1 + \lambda_2 + \lambda_3 + \lambda_4 + \lambda_{\sigma(4)}$. A saddle connection cannot start at $x_2$ and pass through the top of $R_1$ nor can it start at $x_3$ and pass through the top of $R_4$, $R_1$, or $(R_3)$, as $(R_3, R_4, R_1)$ always exists and has smaller slope.\\

\noindent\textbf{Case 4:} ($a_1 < a_3 \leq a_2$). No saddle connections can start at $x_1$ and pass immediately through the top of $R_2$ as $(R_2)$ will always exist and have smaller slope.

\textbf{Case 4.a:} $\pi = (4321), (4132)$. In these cases no saddle connection can pass through the top of $R_4$ nor start at $x_3$. A saddle connection cannot start at $x_1$ and pass through the top of $R_3$ as $(R_2, R_3)$ always exists and has smaller slope, nor can it pass through the top of $R_1$, as either $(R_2)$ will have smaller slope, in the case $\pi = (4321)$, or it will be too long with minimum length $\lambda_2 + \lambda_3 + \lambda_4 + \lambda_1 + \lambda_3 + \lambda_2 > 1$, in the case $\pi = (4132)$. A saddle connection cannot start at $x_2$ and pass immediately through $(R_3)$ as $(R_3, R_2)$ always exists and will have smaller slope, nor can it pass through the top of $(R_1)$ as it will have larger slope than $(R_3,R_2)$, in the case $\pi = (4321)$ or be too long with minimum length $\lambda_3 + \lambda_2 + \lambda_3 + \lambda_4 + \lambda_1$ in the case $\pi = (4132)$.

\textbf{Case 4.b:} $\pi = (3142)$. In this case, no saddle connection can pass through the top of $R_3$ nor start at $x_2$. No saddle connection can start at $x_1$ and pass through the top of $R_4$ as $(R_4, R_2)$ will always exist and have smaller slope, nor can it pass through the top of $R_1$ as this would require a minimum length of $\lambda_2 + \lambda_3 + \lambda_1 + \lambda_4 + \lambda_2 > 1$. A saddle connection cannot start at $x_3$ and pass through the top of either $R_4$ or $R_2$ as it will have larger slope than $(R_4, R_2)$, nor can it pass through the top of $R_1$ as this would require a minimum length of $\lambda_4 + \lambda_2 + \lambda_3 + \lambda_1 + \lambda_4 + \lambda_2 > 1$.\\

\noindent\textbf{Case 5:} ($a_2 \leq a_1 < a_3$). No saddle connection can start at $x_2$ and pass through the top of $R_3$ as $(R_3)$ will always exist and have smaller slope.

\textbf{Case 5.a:} $\pi = (4321), (4213)$. In these cases no saddle connection can pass through the top of $R_4$ nor start at $x_3$. A saddle connection cannot start at $x_1$ and pass through the top of $R_2$ or $R_1$ as it will have larger slope than $(R_1, R_2)$, in addition it cannot pass through the top of $R_3$ as this would require entering from $R_2$ and thus $(R_3)$ will have smaller slope. A saddle connection cannot start at $x_2$ and pass through the top of $R_1$ nor $R_2$ as it will require a minimum length of $\lambda_3 + \lambda_4 + \lambda_2 + \lambda_1 + \lambda_3 > 1$.

\textbf{Case 5.b:} $\pi = (2431), (2413)$. In these cases no saddle connection can pass through the top of $R_2$ or start at $x_1$. If $\pi = (2431)$ then the saddle connection $(R_4, R_3, R_1)$ will always exist and it will have smaller slope than any saddle connection starting at $x_2$ and passing through the top of $R_4$ and $R_1$ as well as any saddle connection starting at $x_3$ and passing through the top of $R_4$, $R_1$, and $R_3$. The same is true in the case of $\pi = (2413)$ with the saddle connection $(R_4, R_3, R_1)$.\\

\noindent\textbf{Case 6:} ($a_3 \leq a_1 < a_2$). No saddle connection can start at $x_1$ and pass through the top of $R_2$ as $(R_2)$ will always exist and have smaller slope.

\textbf{Case 6.a:} $\pi = (4321)$. In this case no saddle connection can pass through the top of $R_4$ nor start at $x_3$. In addition we know that the saddle connection $(R_3, R_2, R_1)$ is always going to exist and will beat any saddle connection starting at $x_1$ and passing through the top of $R_3$ or $R_1$ and any saddle connection starting at $x_2$ and passing through the top of $R_3$, $R_2$, or $R_1$.

\textbf{Case 6.b:} $\pi = (3241)$. In this case no saddle connection can pass through the top of $R_3$ nor start at $x_2$. Any saddle connection starting at $x_1$ and passing through the top of $R_4$ or $R_1$ will have minimum length $\lambda_2 + \lambda_3 + \lambda_1 + \lambda_4 + \lambda_2 > 1$. Any saddle connection starting at $x_3$ and passing through the top of $R_2$ or $R_4$ will have larger slope than $(R_4, R_2)$, and if it passes through the top of $R_1$ it will have minimum length $\lambda_4 + \lambda_2 + \lambda_3 + \lambda_1 + \lambda_4 + \lambda_2 > 1$.

\end{proof}

\subsection{Labeled graph construction}
We can represent every possible saddle connection vector as a path in a labeled tree, where all vertices but the root represent the rectangle, $R_i$, the vector just passed through and the labels are pairs $(a_i - a_{i-1},\lambda_i)$. Branching occurs if we can either pass through to a consecutive rectangle or through to a glued rectangle. The graph depends on the permutation $\pi$ and consists of at most 3 disjoint subgraphs corresponding to the possible starting points.

\begin{example}
Given $\pi = (4321)$ the labeled graph starting at $x_1$ begins as follows:

\begin{center}

\begin{tikzpicture}[
        thick,
        level/.style={level distance=2.5cm, sloped},
        level 2/.style={sibling distance=2.75cm},
        level 3/.style={sibling distance=1.75cm},
        level 4/.style={sibling distance=1cm},
        level 5/.style={sibling distance=.75cm},
        grow=right]
    \node {$x_1$}
    	child {node {$R_2$}{
		child {node {$R_1$}
			child {node {$R_2$}
				child {node {$R_1$}
				edge from parent
				node [below] {\scriptsize $(a_1 - a_0,\lambda_1)$}}
				child {node {$R_3$}
				edge from parent
				node [above] {\scriptsize $(a_3 - a_2,\lambda_3)$}}
			edge from parent
			node [above] {\scriptsize $(a_2 - a_1,\lambda_2)$}}
		edge from parent
		node [above] {\scriptsize $(a_1 - a_0,\lambda_1)$}}
		child {node {$R_3$}
			child {node {$R_2$}
				child {node {$R_1$}
				edge from parent
				node [below] {\scriptsize $(a_1 - a_0,\lambda_1)$}}
				child {node {$R_3$}
				edge from parent
				node [above] {\scriptsize $(a_3 - a_2,\lambda_3)$}}
			edge from parent
			node [below] {\scriptsize $(a_2 - a_1,\lambda_2)$}}
			child {node {$R_4$}
				child {node {$R_3$}
				edge from parent
				node [above] {\scriptsize $(a_3 - a_2,\lambda_3)$}}
			edge from parent
			node [above] {\scriptsize $(a_4 - a_3,\lambda_4)$}}
		edge from parent
		node [above] {\scriptsize $(a_3 - a_2,\lambda_3)$}}
	edge from parent
	node [above] {\scriptsize $(a_2 - a_1,\lambda_2)$}}
	};
\end{tikzpicture}
\end{center}

\end{example}

\begin{note}\label{X0}{\ }
\begin{enumerate}
\item If a vector enters $R_4$ it must continue to $R_{\sigma(4)+1}$, by Note \ref{RED}(3).
\item A vector originating at $x_0$ that passes through a gluing, $R_jR_{\sigma(j)+1}$, before first passing through all rectangles consecutively will never have smallest slope, since the vector connecting $x_0$ to $x_j$ will always exist and have smaller slope. Thus the labeled graph starting at $x_0$ will always begin as follows:
\begin{center}
\begin{tikzpicture}[
        thick,
        level/.style={level distance=1.5cm,sloped},
        level 2/.style={sibling distance=2.25cm},
        level 3/.style={sibling distance=2cm},
        level 4/.style={sibling distance=1.75cm},
        grow=right]
    \node {$x_0$}
        child  {node {$R_1$}{
                  child { node {$R_2$}
            		child{node {$R_3$}
				child{node{$R_4$}
					child{node{$R_{\sigma(4)+1}$}
						edge from parent}
					edge from parent}
				edge from parent}
                    	edge from parent
            }
            edge from parent 
        }}    ;
\end{tikzpicture}
\end{center}
\end{enumerate}
\end{note}

We can denote each path by the sequence of $R_i$ it passes through. For example the path $(R_1,R_2)$ is the path originating at $x_0$ and ending at $x_2$ passing only through consecutive rectangles. In some cases a path may have a repeated structure in which case we use the shorthand $n\times(R_{i_1},\ldots,R_{i_k})$ for a block of size $k$ repeated $n$ times. For example $(R_2, R_3, R_4, R_1, R_4, R_1, R_4, R_1) = (R_2, R_3, 3\times(R_4, R_1))$. Each path has an associated vector where the first coordinate is obtained by summing over the first coordinates associated to its edges and the second coordinate is obtained similarly.

\begin{defn}\label{valid}
We call a path in the graph corresponding to $x_k$ \emph{valid} if it satisfies the following conditions
\begin{enumerate}
\item It has first coordinate $\leq 1$ and positive second coordinate
\item It does not terminate at $x_k$ or $x_4$.
\item It does not cross itself. To ensure this we need to impose the following conditions
\begin{enumerate}
\item If the saddle connection begins with $(R_2, R_3, \ldots)$, or contains $(\ldots, R_{\sigma^{-1}(1)}, R_2, R_3,\ldots)$ then it cannot end with $(\ldots, R_1, R_2)$ nor contain $(\ldots, R_1,R_2,R_{\sigma(2)+1})$, if $\sigma(2) + 1$ exists.
\item If the saddle connection begins with $(R_2, R_3, R_4, \ldots)$ or contains $(\ldots, R_{\sigma^{-1}(1)},R_2, R_3, R_4,\ldots)$ then it cannot end with $(\ldots, R_1, R_2, R_3)$ nor contain $(\ldots, R_1, R_2, R_3, R_{\sigma(3)+1},\ldots)$, if $\sigma(3)+1$ exists.
\item If the saddle connection begins with $(R_3, R_4, \ldots)$ or contains $(\ldots, R_{\sigma^{-1}(2)}, R_3, R_4, \ldots)$ it cannot end with $(\ldots, R_2, R_3)$.
\item If $a_3 > a_1$ the saddle connection cannot wrap around $R_2, R_3$, that is $(n \times(R_2, R_3))$ does not exist for any $n > 1$. Similarly, no saddle connection can wrap around $R_1, R_2$, that is $(n \times (R_1, R_2))$ does not exist for any $n > 1$.
\end{enumerate}
\end{enumerate}
\end{defn}

A valid path will correspond to a saddle connection if it satisfies the existence conditions given in Section \ref{PSRTB}.

\subsection{Algorithm to bound the smallest slope}
\begin{enumerate}
\item Create labeled graphs for all $x_i$ such that $a_i < h_{i+1}$. Recall that we will have $\leq 3$ of these graphs due to Note \ref{RED}(2).
\item Find all paths in the tree whose second coordinate is $\leq 1$
\item Eliminate any paths that start and end at the same $x_i$, i.e. ending at $R_i$, and any paths that end at $x_4$, i.e. ending at $R_4$, or whose first coordinate is $\leq 0$.
\item Find the slopes.
\item Determine if the smallest slope belongs to a legitimate saddle connection, that is, check the existence conditions given in Section \ref{PSRTB}.
\item If the smallest slope does not correspond to a legitimate saddle connection, remove it from the list.
\item Repeat steps 5 \& 6 until a valid saddle connection is obtained.
\end{enumerate}

\subsubsection{Bounding the number of paths}

\begin{note}
The length of any valid path in the tree is bounded by $C = \frac{1}{\min_i(\lambda_i)}$.
\end{note}

\begin{prop}
The maximum number of paths for a given collection of 3 graphs will be bounded above by the quantity 
\[4 + \sum_{i=2}^{C-3}F_i + 2\left(\sum_{i=2}^{C+1}F_i\right)\]
where $F_i$ is the $i$th Fibonacci number.
\end{prop}

\begin{proof}
Note that 4 and $\pi^{-1}(4)$ only have one option to pass to. Suppose that a given graph first branches at level $j \geq 1$, then in all subsequent levels there must be at least one 4 or $\pi^{-1}(4)$ appearing. Consider the case where we have the minimum number of single branches appearing in each level, which would give the maximum number of paths in our graph. The graph is self-similar in structure and contains 3 copies of itself, 2 offset by 2 positions and 1 offset by 3 positions. Let $F_i$ be the number of vertices at level $i$ in the graph, we then get the recursive relation
\[F_i = 2F_{i-2} + F_{i-3}\]
which is equivalent to the recursive relation $F_i = F_{i-1} + F_{i-2}$, for the Fibonacci numbers.

The number of paths in each graph, $N_k$ associated to $a_k$ for $k \in \{0, 1, 2, 3\}$, is bounded by 
\[(j-1) + \sum_{i=2}^{C-(j-2)}F_i.\]
For all $k \neq 0$, we can use $j=1$ to obtain the following bound:
\[N_k \leq \sum_{i=2}^{C+1}F_i.\]
By Note \ref{X0}, when $k = 0$ will have $j \geq 5$, so for all 3 graphs together we obtain a bound of
\[4 + \sum_{i=2}^{C-3}F_i + 2\left(\sum_{i=2}^{C+1}F_i\right)\]
\end{proof}

\begin{example}
Let $\pi = (4321)$ and $C = 5$. The graph corresponding to $x_1$ is shown below
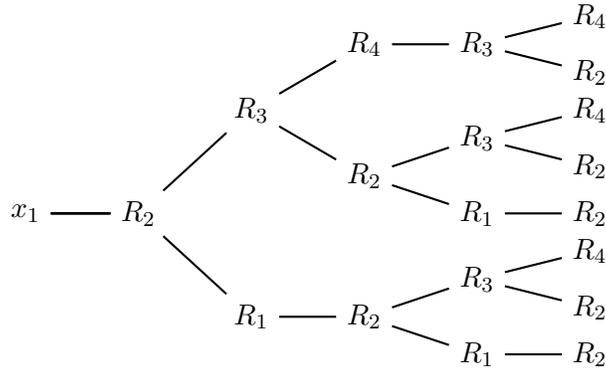
\begin{center}

\begin{tikzpicture}[
        thick,
        level/.style={level distance=1.5cm, sloped},
        level 2/.style={sibling distance=2.75cm},
        level 3/.style={sibling distance=1.75cm},
        level 4/.style={sibling distance=1cm},
        level 5/.style={sibling distance=.75cm},
        grow=right]
    \node {$x_1$}
    	child {node {$R_2$}{
		child {node {$R_1$}
			child {node {$R_2$}
				child {node {$R_1$}
					child {node {$R_2$}
					edge from parent
					}
				edge from parent
				}
				child {node {$R_3$}
					child {node {$R_2$}
					edge from parent
					}
					child {node {$R_4$}
					edge from parent
					}
				edge from parent
				}
			edge from parent
			}
		edge from parent
		}
		child {node {$R_3$}
			child {node {$R_2$}
				child {node {$R_1$}
					child {node {$R_2$}
					edge from parent
					}
				edge from parent
				}
				child {node {$R_3$}
					child {node {$R_2$}
					edge from parent
					}
					child {node {$R_4$}
					edge from parent
					}
				edge from parent
				}
			edge from parent
			}
			child {node {$R_4$}
				child {node {$R_3$}
					child {node {$R_2$}
					edge from parent
					}
					child {node {$R_4$}
					edge from parent
					}
				edge from parent
				}
			edge from parent
			}
		edge from parent
		}
	edge from parent
	}
	};
\end{tikzpicture}
\end{center}

In this example, we branch at $j=1$ and $ N_1 \leq \sum_{i=2}^{6}F_i$.
\end{example}

Since the set of valid paths is finite and non-empty, there will be a path whose corresponding saddle connection has smallest slope, and the algorithm will terminate.

\subsubsection{Examples}
For these illustrative examples, area and lengths have not been normalized.
\begin{example} Consider the zippered rectangles defined by the coordinates $(1, 3; 2, 3, 1, 2)$ and permutation $\pi = (4321)$. For the purposes of this example the total area will be 28 and we will be looking for all paths of length $\leq 8$.
\begin{enumerate}
\item Determine the vectors $a$ and $h$. From example 1, we know $\sigma$ and the equalities, therefore we have
\[a_1 = \frac{28- 1(3) - (3+1)(1) - 3(2)}{2+3} = \frac{15}{5} = 3\]
and so,
\[h = (3, 4, 4, 3) \text{ and } a = (3, 1, 3, 0)\]
\item Since $a_3 = h_3$, there will be no saddle connections emanating from $x_3$, so we need only create the weighted graphs beginning at $x_0, x_1$, and $x_2$. In addition, we have that $a_2 < a_3$ so we do not need to consider any saddle connections that begin with a gluing from $x_2$, and $a_1 \geq a_2$, so we only need to consider saddle connections that begin with a gluing from $x_1$.
\begin{itemize}
\item[$x_0$:]\
\begin{center}

\begin{tikzpicture}[
        thick,
        level/.style={level distance=1.5cm,sloped},
        level 2/.style={sibling distance=2.25cm},
        level 3/.style={sibling distance=2cm},
        level 4/.style={sibling distance=1.75cm},
        grow=right]
    \node {$x_0$}
        child  {node {$R_1$}{
                  child { node {$R_2$}
            	child{node {$R_3$}
			  edge from parent
                    node [above] {$(2,1)$}}
                    edge from parent
			node [above] {$(-2,3)$}
            }
            edge from parent 
            node [above] {$(3,2)$}
        }}    ;
\end{tikzpicture}
\end{center}

\item[$x_1$:]\
\begin{center}

\begin{tikzpicture}[
        thick,
        level/.style={level distance=1.5cm,sloped},
        level 2/.style={sibling distance=2.25cm},
        level 3/.style={sibling distance=2cm},
        level 4/.style={sibling distance=1.5cm},
        grow=right]
    \node {$x_1$}
    	child {node {$R_2$}
		child{node {$R_1$}
			child{node {$R_2$}
				edge from parent
				node[above] {$(-2,3)$}}
			edge from parent
			node [above] {$(3,2)$}}
		edge from parent
		node [above] {$(-2,3)$}}
   ;
\end{tikzpicture}
\end{center}

\item[$x_2$:]\
\begin{center}

\begin{tikzpicture}[
        thick,
        level/.style={level distance=1.5cm,sloped},
        level 2/.style={sibling distance=2.25cm},
        level 3/.style={sibling distance=2cm},
        level 4/.style={sibling distance=1.75cm},
        grow=right]
    \node {$x_2$}
    	child{node {$R_3$}
		child {node {$R_4$}
			child {node {$R_3$}
				child {node {$R_2$}
					child {node {$R_3$}
						edge from parent
						node [above] {$(2,1)$}}
					edge from parent
					node [above] {$(-2,3)$}}
				child {node {$R_4$}
					child {node {$R_3$}
						edge from parent
						node [above] {$(2,1)$}}
					edge from parent
					node [above] {$(-3,2)$}}
				edge from parent
				node [above] {$(2,1)$}}
			edge from parent
			node [above] {$(-3,2)$}}
		edge from parent
		node [above] {$(2,1)$}}
		   ;
\end{tikzpicture}
\end{center}

\end{itemize}

\item Now we need to find all paths in the tree with positive first coordinate and second coordinate $\leq 8$ that do not terminate at $x_4$ or at the starting altitude.
\begin{itemize}
\item[$x_0$:] $(R_1) = \frac{3}{2}, (R_1, R_2) = \frac{1}{5}, (R_1, R_2, R_3) = \frac{3}{6}$
\item[$x_1$:] None
\item[$x_2$:] $(R_2) = \frac{2}{1}, (R_3,R_4,R_3) = \frac{1}{4}, (R_3, R_4,R_3, R_2, R_3) = \frac{1}{8}, (R_3, R_2, R_3) = \frac{2}{5}$
\end{itemize}

\item The smallest slope is $\frac{1}{8}$ so we need to check whether or not $(R_3, R_4,R_3, R_2, R_3)$ is a legitimate saddle connection. We note that the saddle connection is not valid as it violates the condition given in Def. \ref{valid}.3.c. Next smallest is $\frac{1}{5}$ belonging to the vector $(R_1, R_2)$, since $a_1 \geq a_2$, this vector is a saddle connection connecting $x_0$ to $x_2$.
\end{enumerate}

\begin{center}
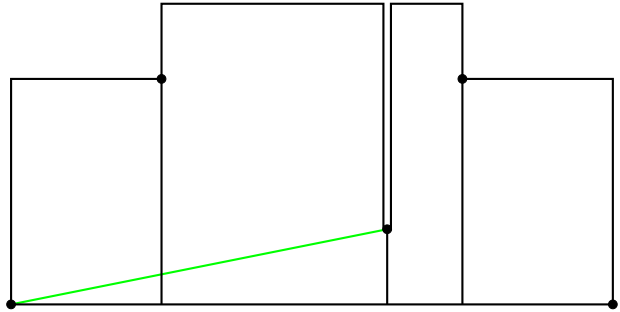
\begin{figure}[h!]\begin{tikzpicture}
\coordinate (A0) at (0,0);
\coordinate (H1) at (0,3);
\coordinate (A1) at (2,3);
\coordinate (H21) at (2,4);
\coordinate (H22) at (5-1/20,4);
\coordinate (A21) at (5-1/20,1);
\coordinate (A2) at (5,1);
\coordinate (A22) at (5+1/20,1);
\coordinate (H31) at (5+1/20,4);
\coordinate (H32) at (6,4);
\coordinate (A3) at (6,3);
\coordinate (H4) at (8,3);
\coordinate (A4) at (8,0);

\coordinate (L1) at (2,0);
\coordinate (L2) at (5,0);
\coordinate (L3) at (6,0);

\draw[thick, green] (A0) -- (A2);

\draw[thick] (A0) -- (H1) -- (A1) -- (H21) -- (H22) -- (A21) -- (A22)-- (H31) -- (H32) -- (A3) -- (H4) -- (A4) -- cycle;
\draw[thick] (L1) -- (A1);
\draw[thick] (L2) -- (A2);
\draw[thick] (L3) -- (A3);

\foreach \A in {A0, A1, A2, A3, A4}
{
\draw[fill=black] (\A) circle (0.15em);
}
\end{tikzpicture}
\caption{The saddle connection of smallest slope}
\label{disk}
\end{figure}
\end{center}

\end{example}

\begin{example}
Consider the zippered rectangles defined by the coordinates $(6, 3; 1, 1, 1, 2)$ and permutation $\pi = (3142)$. For the purposes of this example, total area is 23 and we will be looking for all possible saddle connections of length $\leq 5$.
\begin{enumerate}
\item In order to determine the vectors $a$ and $h$, we must first find the permutation $\sigma$,
\[\sigma(0) = 2, \sigma(1) = 3, \sigma(2) = 4, \sigma(3) = 0, \sigma(4) = 1.\]
The equalities $h_i - a_i = h_{\sigma(1)+1} - a_{\sigma(i)}$ give the following
\begin{alignat*}{ 3}
0 & = h_3 - 6 \hspace{10pt}& \Rightarrow h_3 = 6\\
h_1 - a_ 1 &= h_4 - 3 &\\
h_2 - 6 & = 0 \hspace{10pt} & \Rightarrow h_2 = 6\\
6 - 3 & = h_1 \hspace{10pt} & \Rightarrow h_1 = 3\\
h_4 & = 6 - a_1 & 
\end{alignat*}
Using the total area to solve for $a_1$ yields
\[a_1 = -\frac{23 - 3(1) - 6(1) - 6(1)- 6(2)}{2} = 2\]
Thus, $a = (2, 6, 3, 0)$ and $h = (3, 6, 6, 4)$.
\item Since $a_2 = h_3$, there will be no saddle connections starting from $x_2$. Also, we have that $a_1 < a_2$, so we only need to consider vectors starting from $x_1$ and passing to the next consecutive rectangle. 
\begin{itemize}
\item[$x_0$:]\
\begin{center}

\begin{tikzpicture}[
        thick,
        level/.style={level distance=1.5cm,sloped},
        level 2/.style={sibling distance=2.25cm},
        level 3/.style={sibling distance=2cm},
        level 4/.style={sibling distance=1.75cm},
        grow=right]
    \node {$x_0$}
        child  {node {$R_1$}{
                  child { node {$R_2$}
            	child{node {$R_3$}
			  edge from parent
                    node [above] {$(-3,1)$}}
                    edge from parent
			node [above] {$(4,1)$}
            }
            edge from parent
            node [above] {$(2,1)$}
        }}    ;
\end{tikzpicture}
\end{center}

\item[$x_1$:]\
\begin{center}

\begin{tikzpicture}[
        thick,
        level/.style={level distance=1.5cm,sloped},
        level 2/.style={sibling distance=2.25cm},
        level 3/.style={sibling distance=2cm},
        level 4/.style={sibling distance=1.75cm},
        grow=right]
    \node {$x_1$}
        child  {node {$R_2$}
        		child {node {$R_3$}
			child{node {$R_1$}
				child{node {$R_2$}
					child{node {$R_3$}
						edge from parent
						node[above]{$(-3,1)$}}
					edge from parent
					node[above]{$(4,1)$}}
				edge from parent
				node[above]{$(2,1)$}}
			child{node {$R_4$}
				child{node{$R_2$}
					edge from parent
					node[above]{$(4, 1)$}}
				edge from parent
				node[above]{$(-3,2)$}}
			edge from parent
			node[above]{$(-3,1)$}}
        		edge from parent
		node[above] {$(4,1)$}}
        ;
\end{tikzpicture}
\end{center}

\item[$x_3$:]\
\begin{center}

\begin{tikzpicture}[
        thick,
        level/.style={level distance=1.5cm,sloped},
        level 2/.style={sibling distance=2.25cm},
        level 3/.style={sibling distance=2cm},
        level 4/.style={sibling distance=1.75cm},
        grow=right]
    \node {$x_3$}
        child  {node {$R_4$}
        		child {node {$R_2$}
			child {node {$R_3$}
				child {node {$R_1$}
					edge from parent
					node[above]{$(2,1)$}}
				edge from parent
				node[above]{$(-3,1)$}}
			edge from parent
			node[above]{$(4,1)$}}
        		edge from parent
		node[above] {$(-3,2)$}
        } ;
\end{tikzpicture}
\end{center}

\end{itemize}
\item Now we need to find all paths in the tree with positive first coordinate and second coordinate $\leq 5$ that do not terminate at $x_4$ or at the starting altitude.
\begin{itemize}
\item[$x_0$:] $(R_1) = \frac{2}{1}, (R_1, R_2) = \frac{6}{2}, (R_1, R_2, R_3) = \frac{3}{3}$
\item[$x_1$:] $(R_2) = \frac{4}{1}, (R_2,R_3) = \frac{1}{2}, (R_2, R_3, R_4,R_2) = \frac{2}{5}, (R_2, R_3, R_1,R_2) = \frac{7}{4},  (R_2, R_3, R_1,R_2, R_3) = \frac{4}{5}$
\item[$x_3$:] $(R_4, R_2) = \frac{1}{3}$
\end{itemize}

\item The smallest slope is $\frac{1}{3}$, it is a legitimate saddle connection connecting $x_3$ and $x_2$.
\end{enumerate}

\begin{center}
\begin{figure}[h!]

\begin{tikzpicture}
\coordinate (A0) at (0,0);
\coordinate (H11) at (0,3);
\coordinate (H12) at (1-1/20,3);
\coordinate (A11) at (1-1/20,2);
\coordinate (A1) at (1,2);
\coordinate (A12) at (1+1/20,2);
\coordinate (H2) at (1+1/20,6);
\coordinate (A2) at (2,6);
\coordinate (H3) at (3-1/20,6);
\coordinate (A31) at (3-1/20,3);
\coordinate (A3) at (3,3);
\coordinate (A32) at (3+1/20,3);
\coordinate (H41) at (3+1/20,4);
\coordinate (H42) at (5,4);
\coordinate (A4) at (5,0);

\coordinate (L1) at (1,0);
\coordinate (L2) at (2,0);
\coordinate (L3) at (3,0);

\coordinate (G1) at (5,3+2/3);
\coordinate (G2) at (1+1/20,5+2/3);

\draw[thick, green] (G2) -- (A2);
\draw[thick, green] (A3) -- (G1);

\draw[thick] (A0) -- (H11) -- (H12) -- (A11) -- (A12) -- (H2) -- (A2) -- (H3) -- (A31) -- (A32) -- (H41) -- (H42) -- (A4) -- cycle;
\draw[thick] (L1) -- (A1);
\draw[thick] (L2) -- (A2);
\draw[thick] (L3) -- (A3);

\foreach \A in {A0, A1, A2, A3, A4}
{
\draw[fill=black] (\A) circle (0.15em);
}
\end{tikzpicture}
\caption{The saddle connection of smallest slope}
\label{disk}
\end{figure}
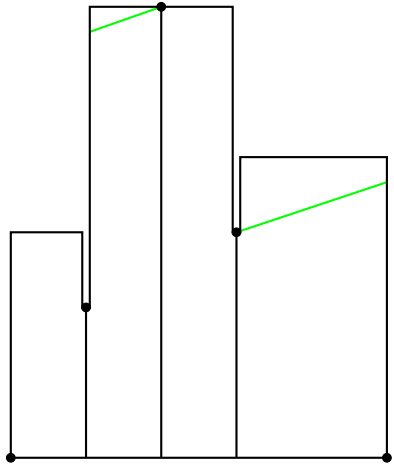
\end{center}

\end{example}

\section{Explicit Description of Poincar\'e Sections and Return Time Functions}\label{PSRTB}

In  this section we give the linear inequalities defining the Poincar\'e section for each of the permutations $\pi$ and bound the return time function on these pieces. To bound the return time we assume $\sum \lambda_i + \min(\lambda_i) > 1$ and let $1-\sum \lambda_i = \epsilon$. We give a complete description for the first two permutations and the rest follow a similar proof.

\subsection{$\pi = (3142)$}
\subsubsection{The Poincar\'e Section}
The equalities that describe the space:
\begin{itemize}
\item $\sum\lambda_i h_i = 1$
\item $h_2 = a_2$
\item $h_3 = a_2$
\item $h_1=h_3-a_3$
\item $h_4 = h_2 - a_1$
\end{itemize}
The inequalities that describe the section:
\begin{itemize}
\item $\sum \lambda_i \leq 1$
\item $\lambda_i > 0$
\item $0 < a_1 < h_1$
\item $a_2 > a_1 + a_3$
\item $0 < a_3 < h_4$
\end{itemize}
\subsubsection{The labeled graph} Since $h_3 = a_2$, no saddle connections can start at $x_2$. Since $\pi(2) = 4$, by Note \ref{RED}(6), a valid saddle connection can only pass from $R_2$ to $R_3$. Thus there will be no branching at any $R_2$ vertex.

\begin{itemize}
\item[$x_0$:]\
\begin{center}

\begin{tikzpicture}[
        thick,
        level/.style={level distance=2.5cm,sloped},
        level 2/.style={sibling distance=2.25cm},
        level 3/.style={sibling distance=2cm},
        level 4/.style={sibling distance=1.75cm},
        grow=right]
    \node {$x_0$}
        child  {node {$R_1$}{
                  child { node {$R_2$}
            	child{node {$R_3$}
			  edge from parent
                    node [above] {$(a_3-a_2,\lambda_3)$}}
                    edge from parent
			node [above] {$(a_2-a_1,\lambda_2)$}
            }
            edge from parent 
            node [above] {$(a_1,\lambda_1)$}
        }}    ;
\end{tikzpicture}
\end{center}

\item[$x_1$:]\
\begin{center}

\begin{tikzpicture}[
        thick,
        level/.style={level distance=2cm, sloped,font=\scriptsize},
        level 2/.style={sibling distance=2.25cm},
        level 3/.style={sibling distance=2.75cm, level distance=1.5cm},
        level 4/.style={sibling distance=1.75cm},
        level 6/.style={level distance=1.75cm,sibling distance=1.4cm},
        level 9/.style={level distance=1cm},
        grow=right]
    \node {$x_1$}
    	child {node {$R_2$}
		child{node {$R_3$}
			child{ node{$R_1$}
				child{node {$R_4$}
					child{node {$R_2$}
						edge from parent
						node[above, font=\scriptsize]{$(a_2-a_1,\lambda_2)$}}
					edge from parent
					node[above]{$(-a_3,\lambda_4)$}}
				child{node {$R_2$}
					child{node{$R_3$}
						child{node{$R_1$}
							child{node{$R_2$}
								child{node{$R_3$}
									child{node{$\cdots$}
										edge from parent}
									edge from parent
									node[above]{$(a_3-a_2,\lambda_3)$}}
								edge from parent
								node[above]{$(a_2-a_1,\lambda_2)$}}
							edge from parent
							node[above]{$(a_1,\lambda_1)$}}
						child{node{$R_4$}
							edge from parent
							node[above]{$(-a_3,\lambda_4)$}}
						edge from parent
						node[above]{$(a_3-a_2,\lambda_3)$}}
					edge from parent
					node[above]{$(a_2-a_1,\lambda_2)$}}
				edge from parent
				node[above]{$(a_1,\lambda_1)$}}
			child{ node{$R_4$}
				child{node{$R_2$}
					child{ node{$R_3$}
						child{node{$R_1$}
							edge from parent
							node[above]{$(a_1,\lambda_1)$}}
						child{node{$R_4$}
							child{node{$R_2$}
								child{node{$R_3$}
									child{node{$\cdots$}
										edge from parent}
									edge from parent
									node[above]{$(a_3-a_2,\lambda_3)$}}
								edge from parent
								node[above]{$(a_2-a_1,\lambda_2)$}}
							edge from parent
							node[above]{$(-a_3,\lambda_4)$}}
						edge from parent
						node[above]{$(a_3-a_2, \lambda_3)$}}
					edge from parent
					node[above]{$(a_2-a_1,\lambda_2)$}}
				edge from parent
				node[above] {$(-a_3, \lambda_4)$}}
			edge from parent
			node [above] {$(a_3-a_2,\lambda_3)$}}
		edge from parent
		node [above] {$(a_2-a_1,\lambda_2)$}}
   ;
\end{tikzpicture}
\end{center}
Note that we can eliminate any branches that result in a second coordinate that has length $> 1$ to obtain the reduced tree diagram

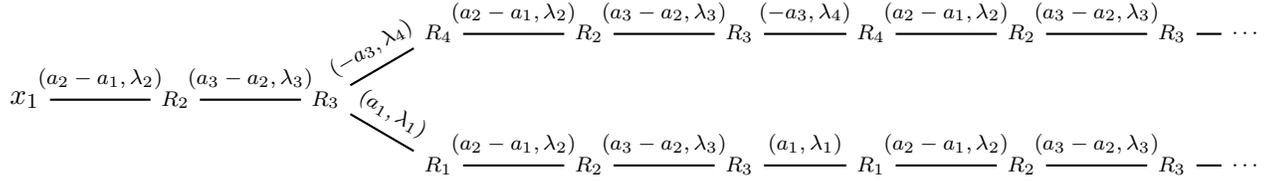
\begin{center}
\begin{tikzpicture}[
        thick,
        level/.style={level distance=2cm, sloped,font=\scriptsize},
        level 2/.style={sibling distance=2.25cm},
        level 3/.style={sibling distance=1.75cm, level distance=1.5cm},
        level 4/.style={sibling distance=2cm},
        level 6/.style={level distance=1.75cm,sibling distance=1.75cm},
        level 9/.style={level distance=1cm},
        grow=right]
    \node {$x_1$}
    	child {node {$R_2$}
		child{node {$R_3$}
			child{ node{$R_1$}
				child{node {$R_2$}
					child{node{$R_3$}
						child{node{$R_1$}
							child{node{$R_2$}
								child{node{$R_3$}
									child{node{$\cdots$}
										edge from parent}
									edge from parent
									node[above]{$(a_3-a_2,\lambda_3)$}}
								edge from parent
								node[above]{$(a_2-a_1,\lambda_2)$}}
							edge from parent
							node[above]{$(a_1,\lambda_1)$}}
						edge from parent
						node[above]{$(a_3-a_2,\lambda_3)$}}
					edge from parent
					node[above]{$(a_2-a_1,\lambda_2)$}}
				edge from parent
				node[above]{$(a_1,\lambda_1)$}}
			child{ node{$R_4$}
				child{node{$R_2$}
					child{ node{$R_3$}
						child{node{$R_4$}
							child{node{$R_2$}
								child{node{$R_3$}
									child{node{$\cdots$}
										edge from parent}
									edge from parent
									node[above]{$(a_3-a_2,\lambda_3)$}}
								edge from parent
								node[above]{$(a_2-a_1,\lambda_2)$}}
							edge from parent
							node[above]{$(-a_3,\lambda_4)$}}
						edge from parent
						node[above]{$(a_3-a_2, \lambda_3)$}}
					edge from parent
					node[above]{$(a_2-a_1,\lambda_2)$}}
				edge from parent
				node[above] {$(-a_3, \lambda_4)$}}
			edge from parent
			node [above] {$(a_3-a_2,\lambda_3)$}}
		edge from parent
		node [above] {$(a_2-a_1,\lambda_2)$}}
   ;
\end{tikzpicture}
\end{center}

Further, note that if $(R_2, R_3)$ exists it will have smaller slope than any longer path following the gluing branch, by Lemma \ref{comp}.1. Thus we can reduce the tree once again

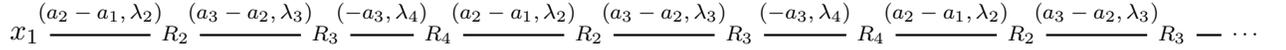
\begin{center}
\begin{tikzpicture}[
        thick,
        level/.style={level distance=2cm, sloped,font=\scriptsize},
        level 2/.style={sibling distance=2.25cm},
        level 3/.style={sibling distance=1.75cm, level distance=1.5cm},
        level 4/.style={sibling distance=2cm},
        level 6/.style={level distance=1.75cm,sibling distance=1.75cm},
        level 9/.style={level distance=1cm},
        grow=right]
    \node {$x_1$}
    	child {node {$R_2$}
		child{node {$R_3$}
			child{ node{$R_4$}
				child{node{$R_2$}
					child{ node{$R_3$}
						child{node{$R_4$}
							child{node{$R_2$}
								child{node{$R_3$}
									child{node{$\cdots$}
										edge from parent}
									edge from parent
									node[above]{$(a_3-a_2,\lambda_3)$}}
								edge from parent
								node[above]{$(a_2-a_1,\lambda_2)$}}
							edge from parent
							node[above]{$(-a_3,\lambda_4)$}}
						edge from parent
						node[above]{$(a_3-a_2, \lambda_3)$}}
					edge from parent
					node[above]{$(a_2-a_1,\lambda_2)$}}
				edge from parent
				node[above] {$(-a_3, \lambda_4)$}}
			edge from parent
			node [above] {$(a_3-a_2,\lambda_3)$}}
		edge from parent
		node [above] {$(a_2-a_1,\lambda_2)$}}
   ;
\end{tikzpicture}
\end{center}

\item[$x_3$:]\
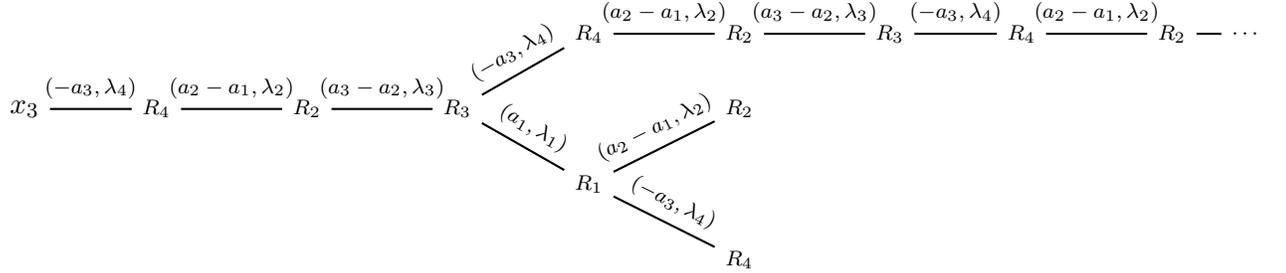
\begin{center}

\begin{tikzpicture}[
        thick,
        level/.style={level distance=2cm, sloped,font=\scriptsize},
        level 1/.style={level distance=1.75cm},
        level 2/.style={sibling distance=2.25cm},
        level 3/.style={sibling distance=2cm},
        level 4/.style={level distance=1.75cm},
        level 7/.style={level distance=1.75cm},
        level 9/.style={level distance=1cm},
        grow=right]
    \node {$x_3$}
    	child{node {$R_4$}
		child {node {$R_2$}
			child{node{$R_3$}
				child{node{$R_1$}
					child{node{$R_4$}
						edge from parent
						node[above]{$(-a_3,\lambda_4)$}}
					child{node{$R_2$}
						edge from parent
						node[above]{$(a_2-a_1,\lambda_2)$}}
					edge from parent
					node[above]{$(a_1, \lambda_1)$}}
				child{node{$R_4$}
					child{node{$R_2$}
						child{node{$R_3$}
							child{node{$R_4$}
								child{node{$R_2$}
									child{node{$\cdots$}
										edge from parent}
									edge from parent
									node[above]{$(a_2-a_1,\lambda_2)$}}
								edge from parent
								node[above]{$(-a_3,\lambda_4)$}}
							edge from parent
							node[above]{$(a_3-a_2,\lambda_3)$}}
						edge from parent
						node[above]{$(a_2-a_1,\lambda_2)$}}
					edge from parent
					node[above]{$(-a_3,\lambda_4)$}}
				edge from parent
				node[above]{$(a_3-a_2,\lambda_3)$}}
			edge from parent
			node[above]{$(a_2-a_1,\lambda_2)$}}
		edge from parent
		node [above] {$(-a_3,\lambda_4)$}}
		   ;
\end{tikzpicture}
\end{center}

Again we can eliminate any paths that result in a second coordinate of length $> 1$ and any paths that end at $R_1$. This results in the reduced tree diagram
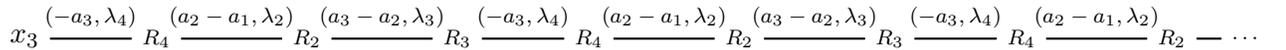
\begin{center}

\begin{tikzpicture}[
        thick,
        level/.style={level distance=2cm, sloped,font=\scriptsize},
        level 1/.style={level distance=1.75cm},
        level 2/.style={sibling distance=2.25cm},
        level 3/.style={sibling distance=2cm},
        level 4/.style={level distance=1.75cm},
        level 7/.style={level distance=1.75cm},
        level 9/.style={level distance=1cm},
        grow=right]
    \node {$x_3$}
    	child{node {$R_4$}
		child {node {$R_2$}
			child{node{$R_3$}
				child{node{$R_4$}
					child{node{$R_2$}
						child{node{$R_3$}
							child{node{$R_4$}
								child{node{$R_2$}
									child{node{$\cdots$}
										edge from parent}
									edge from parent
									node[above]{$(a_2-a_1,\lambda_2)$}}
								edge from parent
								node[above]{$(-a_3,\lambda_4)$}}
							edge from parent
							node[above]{$(a_3-a_2,\lambda_3)$}}
						edge from parent
						node[above]{$(a_2-a_1,\lambda_2)$}}
					edge from parent
					node[above]{$(-a_3,\lambda_4)$}}
				edge from parent
				node[above]{$(a_3-a_2,\lambda_3)$}}
			edge from parent
			node[above]{$(a_2-a_1,\lambda_2)$}}
		edge from parent
		node [above] {$(-a_3,\lambda_4)$}}
		   ;
\end{tikzpicture}
\end{center}

\end{itemize}

Now we consider all possible valid paths in this diagram. 
\[(R_1), (R_1, R_2), (R_1, R_2, R_3), (R_2), (R_2, R_3),(R_4, R_2), \]
\[(R_2, R_3, R_4, R_2), \ldots, (R_2, n\times(R_3, R_4, R_2)),\]
\[(R_2, R_3, R_4, R_2, R_3), \ldots, (R_2, R_3, n\times(R_4, R_2, R_3)),\]
\[ (R_4, R_2, R_3, R_4, R_2), \ldots, (R_4, R_2, n\times(R_3, R_4, R_2))\]

Where $n$ is the maximum number of times the path can repeat before it becomes too long. In each of the last 3 cases the longest valid path will have the smallest slope, let $k$ be the number of times the block is repeated in that path.

Note that $(R_4, R_2)$ always exists, since its first coordinate $a_2 - a_1 - a_3$ will always be positive. This implies that neither $(R_2)$, $(R_1, R_2)$, $(R_2, n\times(R_3, R_4, R_2))$, nor $(R_4, R_2, n\times(R_3, R_4, R_2))$ will have smallest slope, by Lemma \ref{comp}.2. The remaining paths are shown in the table below with their corresponding slopes and existence conditions.

  \begin{center}
\bgroup
\def\arraystretch{2.5}
\begin{tabular}{c | c | c }
Saddle Connection & Slope & Existence Condition(s)\\
\hline
$(R_1)$ & $\dfrac{a_1}{\lambda_1}$ & None\\
$(R_1, R_2, R_3)$ & $\dfrac{a_3}{\lambda_1+\lambda_2+\lambda_3}$ & $a_1 > \dfrac{\lambda_1(a_2)}{\lambda_1+\lambda_2}$, $a_2 > \dfrac{(\lambda_1+\lambda_2)a_3}{\lambda_1+\lambda_2+\lambda_3}$ \\
$(R_2, R_3)$ & $\dfrac{a_3 -a_1}{\lambda_2 + \lambda_3}$ & $a_3 > a_1$\\
$(R_2, R_3, k\times(R_4, R_2, R_3))$ & $\dfrac{a_3 - (k+1)a_1}{(k+1)(\lambda_2+\lambda_3)+k\lambda_4}$ &  \parbox{7cm}{\vspace{.75\baselineskip} \centering$a_3 > (k+1)a_1$,\\ $a_1 + \dfrac{(\lambda_2+\lambda_3)(a_3 - (k+1)a_1)}{(k+1)(\lambda_2 + \lambda_3) + k\lambda_4} < a_3$,\\ $(k+1)(\lambda_2+\lambda_3)+k\lambda_4 \leq 1$} \\
$(R_4, R_2)$ & $\dfrac{a_2 - (a_1+a_3)}{\lambda_2+\lambda_4}$ & None\\
\end{tabular}
\egroup
\end{center}

First assume $a_1 \geq a_3$. In this case, since $a_2 > a_1$ and $a_2 > a_3$, $(R_1, R_2, R_3)$ will always exist and thus, $(R_1)$ will never have smallest slope, by Lemma \ref{comp}.1. $(R_1, R_2, R_3)$ has smallest slope if
\[\frac{a_3}{\lambda_1+\lambda_2+\lambda_3} \leq \frac{a_2-(a_1+a_3)}{\lambda_4+\lambda_2}\]
otherwise $(R_4, R_2)$ has smallest slope.

Now assume $a_1 < a_3$. In this case, $(R_2, R_3)$ always exists. If $(R_2, R_3, k\times(R_4, R_2, R_3))$ exists, it will beat $(R_2, R_3)$. So assume it does not exist. 
If we also have that
\[\frac{a_1}{\lambda_1} \leq \frac{a_3-a_1}{\lambda_2+\lambda_3}\]
then $(R_1)$ has smallest slope. If not, then $(R_2, R_3)$ has smallest slope if
\[\frac{a_3-a_1}{\lambda_2+\lambda_3} \leq \frac{a_2-(a_1+a_3)}{\lambda_4+\lambda_2}\]
otherwise $(R_4,R_2)$ has smallest slope.

\subsection{$\pi = (3241)$}
\subsubsection{The Poincar\'e Section}
The equalities that describe the space:
\begin{itemize}
\item $\sum\lambda_i h_i = 1$
\item $h_1 = a_1$
\item $h_3 = a_2$
\item $h_4=h_1$
\item $h_3-a_3 = h_2 - a_1$
\end{itemize}
The inequalities that describe the section:
\begin{itemize}
\item $\sum \lambda_i \leq 1$
\item $\lambda_i > 0$
\item $0 < a_3 <  a_2 < a_1 \leq h_2$
\end{itemize}
\subsubsection{Bounding the return times} Since $h_3=a_2$, no saddle connections can start at $x_2$, and since $\pi(1) = 4$, a valid saddle connection can only pass from $R_1$ to $R_2$, thus there will be no branching at any $R_1$ vertex.
\begin{itemize}
\item[$x_0$:]\
\begin{center}

\begin{tikzpicture}[
        thick,
        level/.style={level distance=2.5cm,sloped},
        level 2/.style={sibling distance=2.25cm},
        level 3/.style={sibling distance=2cm},
        level 4/.style={sibling distance=1.75cm},
        grow=right]
    \node {$x_0$}
        child  {node {$R_1$}{
                  child { node {$R_2$}
            	child{node {$R_3$}
			  edge from parent
                    node [above] {$(a_3-a_2,\lambda_3)$}}
                    edge from parent
			node [above] {$(a_2-a_1,\lambda_2)$}
            }
            edge from parent 
            node [above] {$(a_1,\lambda_1)$}
        }}    ;
\end{tikzpicture}
\end{center}

\item[$x_1$:]\
\begin{center}

\begin{tikzpicture}[
        thick,
        level/.style={level distance=2cm,sloped,font=\scriptsize},
        level 2/.style={sibling distance=2.25cm},
        level 3/.style={sibling distance=2.75cm},
        level 4/.style={sibling distance=1.5cm},
        level 5/.style={sibling distance =2.25cm},
        level 8/.style={level distance=1cm},
        grow=right]
    \node {$x_1$}
    	child {node {$R_2$}
		child{node {$R_4$}
			child{ node{$R_1$}
				child{node{$R_2$}
					child{node{$R_4$}
						child{node{$R_1$}	
							child{node{$R_2$}
								child{node{$\cdots$}
									edge from parent}
								edge from parent
								node[above]{$(a_2-a_1,\lambda_2)$}}
							edge from parent
							node[above]{$(a_1, \lambda_1)$}}
						edge from parent
						node[above]{$(-a_3, \lambda_4)$}}
					child{node{$R_3$}
						edge from parent
						node[above]{$(a_3-a_2, \lambda_3)$}}
					edge from parent
					node[above]{$(a_2-a_1,\lambda_2)$}}
				edge from parent
				node[above] {$(a_1, \lambda_1)$}}
			edge from parent
			node[above] {$(-a_3,\lambda_4)$}}
		edge from parent
		node [above] {$(a_2-a_1,\lambda_2)$}}
   ;
\end{tikzpicture}
\end{center}

Eliminating branches that result in second coordinates greater than $1$ yields the following reduced tree diagram

\begin{center}

\begin{tikzpicture}[
        thick,
        level/.style={level distance=2cm,sloped,font=\scriptsize},
        level 2/.style={sibling distance=2.25cm},
        level 3/.style={sibling distance=2.75cm},
        level 4/.style={sibling distance=1.5cm},
        level 5/.style={sibling distance =2.25cm},
        level 8/.style={level distance=1cm},
        grow=right]
    \node {$x_1$}
    	child {node {$R_2$}
		child{node {$R_4$}
			child{ node{$R_1$}
				child{node{$R_2$}
					child{node{$R_4$}
						child{node{$R_1$}	
							child{node{$R_2$}
								child{node{$\cdots$}
									edge from parent}
								edge from parent
								node[above]{$(a_2-a_1,\lambda_2)$}}
							edge from parent
							node[above]{$(a_1, \lambda_1)$}}
						edge from parent
						node[above]{$(-a_3, \lambda_4)$}}
					edge from parent
					node[above]{$(a_2-a_1,\lambda_2)$}}
				edge from parent
				node[above] {$(a_1, \lambda_1)$}}
			edge from parent
			node[above] {$(-a_3,\lambda_4)$}}
		edge from parent
		node [above] {$(a_2-a_1,\lambda_2)$}}
   ;
\end{tikzpicture}
\end{center}

\item[$x_3$:]\
\begin{center}

\begin{tikzpicture}[
        thick,
        level/.style={level distance=2cm,sloped,font=\scriptsize},
        level 2/.style={sibling distance=2.25cm},
        level 3/.style={sibling distance=2cm},
        level 4/.style={sibling distance=2cm},
        level 5/.style={sibling distance=2.25cm},
        level 7/.style={level distance=1cm},
        grow=right]
    \node {$x_3$}
    	child{node {$R_4$}
		child {node {$R_1$}
			child {node {$R_2$}
				child{node{$R_4$}
					child{node{$R_1$}
						child{node{$R_2$}
							child{node{$\cdots$}
								edge from parent}
							edge from parent
							node[above]{$(a_2-a_1,\lambda_2)$}}
						edge from parent
						node[above]{$(a_1,\lambda_1)$}}
					edge from parent
					node[above]{$(-a_3,\lambda_4)$}}
				child{node{$R_3$}
					child{node{$R_2$}
						edge from parent
						node[above]{$(a_2-a_1,\lambda_2)$}}
					child{node{$R_4$}
						child{node{$R_1$}
							edge from parent
							node[above]{$(a_1,\lambda_1)$}}
						edge from parent
						node[above]{$(-a_3,\lambda_4)$}}
					edge from parent
					node[above]{$(a_3-a_2,\lambda_3)$}}
				edge from parent
				node[above]{$(a_2-a_1, \lambda_2)$}}
			edge from parent
			node[above]{$(a_1,\lambda_1)$}}
		edge from parent
		node [above] {$(-a_3,\lambda_4)$}}
		   ;
\end{tikzpicture}
\end{center}

Which results in the reduced tree diagram

\begin{center}

\begin{tikzpicture}[
        thick,
        level/.style={level distance=2cm,sloped,font=\scriptsize},
        level 2/.style={sibling distance=2.25cm},
        level 3/.style={sibling distance=2cm},
        level 4/.style={sibling distance=2cm},
        level 5/.style={sibling distance=2.25cm},
        level 7/.style={level distance=1cm},
        grow=right]
    \node {$x_3$}
    	child{node {$R_4$}
		child {node {$R_1$}
			child {node {$R_2$}
				child{node{$R_4$}
					child{node{$R_1$}
						child{node{$R_2$}
							child{node{$\cdots$}
								edge from parent}
							edge from parent
							node[above]{$(a_2-a_1,\lambda_2)$}}
						edge from parent
						node[above]{$(a_1,\lambda_1)$}}
					edge from parent
					node[above]{$(-a_3,\lambda_4)$}}
				edge from parent
				node[above]{$(a_2-a_1, \lambda_2)$}}
			edge from parent
			node[above]{$(a_1,\lambda_1)$}}
		edge from parent
		node [above] {$(-a_3,\lambda_4)$}}
		   ;
\end{tikzpicture}
\end{center}

Here we see that the path $(R_4, R_1, R_2)$ will always be a valid path and will have smaller slope than any longer path in this tree, therefore we can reduce it further to 
\begin{center}

\begin{tikzpicture}[
        thick,
        level/.style={level distance=2.5cm,sloped},
        level 2/.style={sibling distance=2.25cm},
        level 3/.style={sibling distance=2cm},
        level 4/.style={sibling distance=2cm},
        level 5/.style={sibling distance=2.25cm},
        level 7/.style={level distance=1cm},
        grow=right]
    \node {$x_3$}
    	child{node {$R_4$}
		child {node {$R_1$}
			child {node {$R_2$}
				edge from parent
				node[above]{$(a_2-a_1, \lambda_2)$}}
			edge from parent
			node[above]{$(a_1,\lambda_1)$}}
		edge from parent
		node [above] {$(-a_3,\lambda_4)$}}
		   ;
\end{tikzpicture}
\end{center}

\end{itemize}

Now we consider all possible valid paths in the reduced diagram
\[(R_1), (R_1,R_2), (R_1, R_2, R_3), (R_4, R_1),(R_4, R_1, R_2),\]
\[(R_2, R_4, R_1, R_2), \ldots, (R_2, n \times(R_4, R_1, R_2))\]
In the last sequence the shortest valid path will have the smallest slope, let $(R_2, k \times(R_4, R_1, R_2))$ be that path, where $1 \leq k \leq n$.

Since $(R_1, R_2, R_3)$ always exits, neither $(R_1)$ nor $(R_1, R_2)$ will have have smallest slope, by Lemma \ref{comp}.1. Similarly, $(R_4, R_1)$ will never have smallest slope since $(R_4, R_1, R_2)$ always exists and will have smaller slope.

The remaining paths are shown in the table below with their corresponding slopes and existence conditions

  \begin{center}
\bgroup
\def\arraystretch{2.25}
\begin{tabular}{c | c | c }
Saddle Connection & Slope & Existence Condition(s)\\
\hline
$(R_1, R_2, R_3)$ & $\dfrac{a_3}{\lambda_1+\lambda_2+\lambda_3}$ & None \\
$(R_2, k \times (R_4, R_1, R_2))$ & $\dfrac{(k+1)a_2 - ka_3 - a_1}{k\lambda_1 + (k+1)\lambda_2+k\lambda_4}$ & \parbox{7cm}{\vspace{.75\baselineskip}\centering $(k+1)a_2 > a_1+ka_3$, \\ $k\lambda_1 + (k+1)\lambda_2+k\lambda_4 \leq 1$}\\
$(R_4, R_1, R_2)$ & $\dfrac{a_2-a_3}{\lambda_1 + \lambda_2 + \lambda_3}$ & None\\
\end{tabular}
\egroup
\end{center}

If $(R_2, k \times (R_4, R_1, R_2))$ exists it will have smallest slope if
\[\frac{(k+1)a_2 - ka_3 - a_1}{k\lambda_1 + (k+1)\lambda_2+k\lambda_4} < \frac{a_3}{\lambda_1+\lambda_2+\lambda_3}\]
otherwise $(R_1, R_2, R_3)$ will have smallest slope. If it does not exist, then $(R_1, R_2, R_3)$ will have smallest slope if
\[\frac{a_3}{\lambda_1+\lambda_2+\lambda_3} \leq \frac{a_2-a_3}{\lambda_1 + \lambda_2 + \lambda_3}\]
otherwise $(R_4, R_1, R_2)$ will have smallest slope.

\subsection{$\pi = (4132)$}

\subsubsection{The Poincar\'e Section}
The equalities that describe the space:
\begin{itemize}
\item $\sum\lambda_i h_i = 1$
\item $h_4 = a_3$
\item $h_2 = a_2$
\item $h_4 = h_1$
\item $h_1 - a_1 = h_3 - a_2$
\end{itemize}
The inequalities that describe the section:
\begin{itemize}
\item $\sum \lambda_i \leq 1$
\item $\lambda_i > 0$
\item $0 < a_1 < a_2, a_3$
\item $a_2 \leq h_3$
\item $a_3 \leq h_3$
\end{itemize}

\subsubsection{Bounding the return times}
Since $h_4 = a_3$, no saddle connections can start from $x_3$. In addition, all $a_i$ must be distinct, otherwise our transversal will not be a saddle connection as it will contain a singular point in its interior. The reduced tree diagram giving the candidate saddle connections is shown below:
 
 \begin{itemize}
 \bgroup
\item[$x_0$:]\
\begin{center}

\begin{tikzpicture}[
        thick,
        level/.style={level distance=2.5cm,sloped},
        level 2/.style={sibling distance=2.25cm},
        level 3/.style={sibling distance=2cm},
        level 4/.style={sibling distance=1.75cm},
        grow=right]
    \node {$x_0$}
        child  {node {$R_1$}{
                  child { node {$R_2$}
            	child{node {$R_3$}
			  edge from parent
                    node [above] {$(a_3-a_2,\lambda_3)$}}
                    edge from parent
			node [above] {$(a_2-a_1,\lambda_2)$}
            }
            edge from parent 
            node [above] {$(a_1,\lambda_1)$}
        }}    ;
\end{tikzpicture}
\end{center}
\egroup

\bgroup
\item[$x_1$:]\
\begin{center}

\begin{tikzpicture}[
        thick,
        level/.style={level distance=2.5cm,sloped},
        level 2/.style={sibling distance=2.25cm},
        level 4/.style={level distance=1cm},
        grow=right]
    \node {$x_1$}
    	child {node {$R_2$}
		child{node {$R_3$}
			edge from parent
			node [above] {$(a_3-a_2,\lambda_3)$}}
		edge from parent
		node [above] {$(a_2-a_1,\lambda_2)$}}
   ;
\end{tikzpicture}
\end{center}
\egroup

\bgroup
\item[$x_2$:]\
\begin{center}

\begin{tikzpicture}[
        thick,
        level/.style={level distance=2.5cm,sloped},
        level 2/.style={sibling distance=2.75cm},
        level 3/.style={sibling distance=2cm},
        level 5/.style={level distance=1cm},
        grow=right]
    \node {$x_2$}
    	child{node {$R_3$}
		child{node{$R_2$}
			child{node{$R_3$}
				child{node{$R_2$}
					child{node{$\cdots$}}
					edge from parent
					node[above]{$(a_2-a_1, \lambda_2)$}}
				edge from parent
				node[above]{$(a_3-a_2, \lambda_3)$}}
			edge from parent
			node[above]{$(a_2-a_1, \lambda_2)$}}
		child{node{$R_4$}
			child{node{$R_1$}
				child{node{$R_3$}
					child{node{$\cdots$}}
					edge from parent
					node[above]{$(a_3-a_2,\lambda_3)$}}
				edge from parent
				node[above]{$(a_1,\lambda_1)$}}
			edge from parent
			node[above]{$(-a_3,\lambda_4)$}}
		edge from parent
		node [above] {$(a_3-a_2,\lambda_3)$}}
		   ;
\end{tikzpicture}
\end{center}
\egroup
\end{itemize}

The possible valid paths are as follows
\[(R_1), (R_1, R_2), (R_1, R_2, R_3), (R_2),(R_3),(R_2, R_3),\]
\[(R_3, R_4,R_1, R_3),\ldots, (R_3, n\times(R_4, R_1, R_3)),\]
 \[(R_3, R_2, R_3), \ldots, (R_3, n \times (R_2, R_3))\]

We need only consider $(R_3, k \times (R_4, R_1, R_2))$ corresponding to the longest valid path in its corresponding sequences, and $(R_3, k \times (R_2, R_3)$ corresponding to the shortest valid path in its sequence. In addition, since $(R_2)$ always exists, $(R_1, R_2)$ will never have smallest slope, by Lemma \ref{comp}.2.

  \begin{center}
\bgroup
\def\arraystretch{2.25}
\begin{tabular}{c | c | c }
Saddle Connection & Slope & Existence Condition(s)\\
\hline
$(R_1)$ & $\dfrac{a_1}{\lambda_1}$ & None\\
$(R_1, R_2, R_3)$ & $\dfrac{a_3}{\lambda_1+\lambda_2+\lambda_3}$ & $a_1 > \dfrac{\lambda_1(a_2)}{\lambda_1+\lambda_2}$, $a_2 > \dfrac{(\lambda_1+\lambda_2)a_3}{\lambda_1+\lambda_2+\lambda_3}$ \\
$(R_2)$ & $\dfrac{a_2 - a_1}{\lambda_1}$ & None\\
$(R_2, R_3)$ & $\dfrac{a_3 -a_1}{\lambda_2 + \lambda_3}$ & $a_2 > \dfrac{\lambda_2(a_3 - a_1)}{\lambda_2+\lambda_1} + a_1$\\
$(R_3)$ & $\dfrac{a_3 - a_2 }{\lambda_3}$ & $a_3 > a_2$\\
$(R_3, k \times(R_4, R_1, R_3))$ & $\dfrac{ka_1 +a_3- (k+1)a_2}{k\lambda_1 + (k+1)\lambda_3 + k\lambda_4}$ & \parbox{5cm}{\centering $a_3 > a_2 +k(a_2-a_1)$\\ $k\lambda_1 + (k+1)\lambda_3 + k\lambda_4 \leq 1$}\\
$(R_3, k \times( R_2, R_3))$ & $\dfrac{(k+1)a_3 - a_2 - ka_1}{k\lambda_2 + (k+1)\lambda_3}$ & \parbox{5cm}{\centering$(k+1)a_3 - a_2 - ka_1 > 0$\\$k\lambda_2 + (k+1)\lambda_3 \leq 1$}\\
\end{tabular}
\egroup
\end{center}

First assume $a_3 \leq a_2$. In this case $(R_2, R_3)$ always exists and beats $(R_2)$ and $(R_1, R_2, R_3)$. Thus the candidates for smallest slope are $(R_1)$, $(R_2, R_3)$, and $(R_3, k \times(R_2, R_3))$.

Next assume $a_3 > a_2$. In this case $(R_3)$ always exist and beats $(R_1, R_2, R_3)$, $(R_2, R_3)$, and $(R_3, k \times( R_2, R_3))$. Thus the candidates for smallest slope are $(R_1)$, $(R_2)$, $(R_3)$, and $(R_3, k \times(R_4, R_1, R_3))$.

\subsection{$\pi = (2413)$}
\subsubsection{The Poincar\'e Section}
The equalities that describe the space:
\begin{itemize}
\item $\sum\lambda_i h_i = 1$
\item $h_2 = a_1$
\item $h_3 = a_3$
\item $h_4 = h_1$
\item $h_1 - a_1 = h_3 - a_2$
\end{itemize}
The inequalities that describe the section:
\begin{itemize}
\item $\sum \lambda_i \leq 1$
\item $\lambda_i > 0$
\item $a_1 \leq h_1$
\item $0 < a_2 < a_1$
\item $a_2 \leq a_3$
\item $a_3 \leq h_1$
\end{itemize}
\subsubsection{Bounding the return times}
Since $h_2 = a_1$, no saddle connections can start at $x_1$. For the transversal to be a saddle connection $a_1 \neq a_2$. 
 
  \begin{itemize}
\item[$x_0$:]\
\begin{center}

\begin{tikzpicture}[
        thick,
        level/.style={level distance=2.5cm,sloped},
        level 2/.style={sibling distance=2.25cm},
        level 3/.style={sibling distance=2cm},
        level 4/.style={sibling distance=1.75cm},
        grow=right]
    \node {$x_0$}
        child  {node {$R_1$}{
                  child { node {$R_2$}
            	child{node {$R_3$}
			  edge from parent
                    node [above] {$(a_3-a_2,\lambda_3)$}}
                    edge from parent
			node [above] {$(a_2-a_1,\lambda_2)$}
            }
            edge from parent 
            node [above] {$(a_1,\lambda_1)$}
        }}    ;
\end{tikzpicture}
\end{center}

\item[$x_2$:]\
\begin{center}

\begin{tikzpicture}[
        thick,
        level/.style={level distance=2.5cm,sloped},
        level 2/.style={sibling distance=2.25cm},
        level 3/.style={sibling distance=2.75cm},
        level 4/.style={sibling distance=1.5cm},
        level 5/.style={level distance=1cm},
        grow=right]
    \node {$x_2$}
    	child {node {$R_3$}
		child{node {$R_4$}
			child{ node{$R_1$}
				child{node{$R_3$}
					child{node{$\cdots$}}
					edge from parent
					node[above]{$(a_3-a_2, \lambda_3)$}}
				edge from parent
				node[above] {$(a_1,\lambda_1)$}}
			edge from parent
			node [above] {$(-a_3,\lambda_4)$}}
		edge from parent
		node [above] {$(a_3-a_2,\lambda_3)$}}
   ;
\end{tikzpicture}
\end{center}

\item[$x_3$:]\
\begin{center}

\begin{tikzpicture}[
        thick,
        level/.style={level distance=2.5cm,sloped},
        level 2/.style={sibling distance=2.75cm},
        level 3/.style={sibling distance=2.5cm},
        level 4/.style={sibling distance=1.75cm},
        level 6/.style={level distance=1cm},
        grow=right]
    \node {$x_3$}
    	child{node {$R_4$}
		child{node{$R_1$}
			child{node{$R_2$}
				child{node{$R_4$}
					child{node{$R_1$}
						child{node{$\cdots$}}
						edge from parent
						node[above]{$(a_1, \lambda_1)$}}
					edge from parent
					node[above]{$(-a_3, \lambda_3)$}}
				edge from parent
				node[above]{$(a_2-a_1, \lambda_2)$}}
			child{node{$R_3$}
				child{node{$R_4$}
					child{node{$R_1$}
						child{node{$\cdots$}}
						edge from parent
						node[above]{$(a_1, \lambda_1)$}}
					edge from parent
					node[above]{$(-a_3, \lambda_4)$}}
				edge from parent
				node[above]{$(a_3-a_2, \lambda_3)$}}
			edge from parent
			node[above]{$(a_1,\lambda_1)$}}
		edge from parent
		node [above] {$(-a_3,\lambda_4)$}}
		   ;
\end{tikzpicture}
\end{center}

\end{itemize}

The possible valid paths are as follows
\[(R_1), (R_1, R_2), (R_1, R_2, R_3), (R_3), (R_4, R_1),\]
\[(R_3, R_4, R_1), \ldots, (n\times(R_3, R_4, R_1)),\]
\[(R_3, R_4, R_1, R_3), \ldots, (R_3, n \times(R_4, R_1, R_3)),\]
\[(R_4, R_1, R_2, R_4, R_1), \ldots, (R_4, R_1, n\times(R_2, R_4, R_1)),\]
\[(R_4, R_1, R_3, R_4, R_1), \ldots, (R_4, R_1, n\times (R_3, R_4, R_1))\]

 Note that if $(R_3, R_4, R_1)$ does not exist, than neither do any of the other paths in the sequence, and if it does exist it will have the smallest slope. We also only need to consider $(R_4, R_1, k \times(R_3, R_4, R_1))$ corresponding to the shortest valid path in its sequence and $(R_3, k \times(R_4, R_1, R_3)$ and $(R_4, R_1, k\times (R_2, R_4, R_1))$ corresponding to the longest valid path in their respective sequences. In addition, since $(R_1, R_2)$ always exists, $(R_1)$ will never have smallest slope, and since $(R_3)$ always exists $(R_1, R_2, R_3)$ will never have smallest slope.

 \begin{center}
\bgroup
\def\arraystretch{3.5}
\begin{tabular}{c | c | c}
Saddle Connection & Slope & Existence Condition(s)\\
\hline
$(R_1, R_2)$ & $\dfrac{a_2}{\lambda_1+\lambda_2}$ & None\\
$(R_3)$ & $\dfrac{a_3 - a_2 }{\lambda_3}$ & None\\
$(R_3, R_4, R_1)$ & $\dfrac{a_1 - a_2}{\lambda_1 + \lambda_2}$ & $a_3 > \dfrac{\lambda_3(a_1-a_2)}{\lambda_3+\lambda_4+\lambda_1} + a_1$\\
$(R_3, k \times(R_4, R_1, R_3))$ & $\dfrac{a_3+ka_1-(k+1)a_2}{(k+1)\lambda_3 + k(\lambda_1+\lambda_4)}$ &\parbox{7cm}{\centering{ $a_3 > a_2 + k(a_2 - a_1)$\\$(k+1)\lambda_3 + k(\lambda_1+\lambda_4) \leq 1$}}\\
$(R_4, R_1)$ & $\dfrac{a_1- a_3}{\lambda_1 + \lambda_4}$ & $a_1 > a_3$\\
$(R_4, R_1, k\times (R_2, R_4, R_1))$ & $\dfrac{a_1+ka_2-(k+1)a_3}{(k+1)(\lambda_4 +\lambda_1) + k\lambda_2}$ & \parbox{7cm}{\centering{ $a_1 > a_3 + k(a_3 - a_2)$\\$a_3 + \dfrac{(\lambda_4+\lambda_4)(a_1+ka_2-(k+1)a_3)}{(k+1)(\lambda_4 +\lambda_1) + k\lambda_2} < a_1$\\$(k+1)(\lambda_4 +\lambda_1) + k\lambda_2\leq 1$}}\\
$(R_4, R_1, k \times(R_3, R_4, R_1))$ & $\dfrac{(k+1)a_1 - ka_2 - a_3}{(k+1)(\lambda_4+\lambda_1)+k\lambda_3}$ & \parbox{7cm}{\centering{ $a_3 < a_1 + k(a_1-a_2)$\\$a_3 + \dfrac{(\lambda_4+\lambda_1)((k+1)a_1 - ka_2 - a_3)}{(k+1)(\lambda_4+\lambda_1)+k\lambda_3} < a_1$\\$(k+1)(\lambda_4+\lambda_1)+k\lambda_3\leq 1$}}\\
\end{tabular}
\egroup
\end{center}

Assume $a_1 > a_3$, $(R_4, R_1)$ always exists and it beats $(R_3, R_4, R_1)$ and $(R_4, R_1, k\times(R_3,R_4,R_1))$. Thus the candidates for smallest slope are $(R_1, R_2)$, $(R_3)$, $(R_3, k \times(R_4, R_1, R_3))$, $(R_4, R_1)$, and $(R_4, R_1, k \times(R_2, R_4, R_1))$.

Next assume $a_1 \leq a_3$. The candidates for smallest slope are $(R_1, R_2)$, $(R_3)$, $(R_3, R_4, R_1)$, $(R_3, k \times(R_4, R_1, R_3))$, and $(R_4, R_1, k \times(R_3, R_4, R_1))$.

\subsection{$\pi = (2431)$}
\subsubsection{The Poincar\'e Section}
The equalities that describe the space:
\begin{itemize}
\item $\sum\lambda_i h_i = 1$
\item $h_2 = a_1$
\item $h_1 = a_1$
\item $h_1 = h_3 - a_3$
\item $h_4 = h_3 - a_2$
\end{itemize}
The inequalities that describe the section:
\begin{itemize}
\item $\sum \lambda_i \leq 1$
\item $\lambda_i > 0$
\item $0 < a_2 < a_1$
\item $0 < a_3 \leq h_4$
\end{itemize}
\subsubsection{Bounding the return times} Since $a_1=h_2$, no saddle connections can start from $x_1$.

 \begin{itemize}
  
\bgroup 
\item[$x_0$:]\
\begin{center}

\begin{tikzpicture}[
        thick,
        level/.style={level distance=2.5cm,sloped},
        level 2/.style={sibling distance=2.25cm},
        level 3/.style={sibling distance=2cm},
        level 4/.style={sibling distance=1.75cm},
        grow=right]
    \node {$x_0$}
        child  {node {$R_1$}{
                  child { node {$R_2$}
            	child{node {$R_3$}
			  edge from parent
                    node [above] {$(a_3-a_2,\lambda_3)$}}
                    edge from parent
			node [above] {$(a_2-a_1,\lambda_2)$}
            }
            edge from parent 
            node [above] {$(a_1,\lambda_1)$}
        }}    ;
\end{tikzpicture}
\end{center}
\egroup

\bgroup
\item[$x_2$:]\
\begin{center}

\begin{tikzpicture}[
        thick,
        level/.style={level distance=2.5cm,sloped},
        level 2/.style={sibling distance=2.25cm},
        level 3/.style={sibling distance=2.75cm},
        level 4/.style={sibling distance=2cm},
        grow=right]
    \node {$x_2$}
    	child {node {$R_3$}
		child{node{$R_1$}
			child{node{$R_2$}
				child {node{$R_3$}
					child{node{$\cdots$}}
					edge from parent
					node [above] {$(a_3-a_2,\lambda_3)$}}
				edge from parent
				node [above] {$(a_2-a_1,\lambda_2)$}}
			edge from parent
			node [above] {$(a_1, \lambda_1)$}}
		child{node {$R_4$}
			child{ node{$R_3$}
				child { node{$R_1$}
					edge from parent
					node [above] {$(a_1, \lambda_1)$}}
				child{node{$R_4$}
					child { node{$R_3$}
						child{node{$R_1$}
							edge from parent
							node[above]{$(a_1, \lambda_1)$}}
						child {node{$\cdots$}}
						edge from parent
						node[above]{$(a_3-a_2,\lambda_3)$}}
					edge from parent
					node [above] {$(-a_3, \lambda_4)$}}
				edge from parent
				node[above] {$(a_3-a_2,\lambda_3)$}}
			edge from parent
			node [above] {$(-a_3,\lambda_4)$}}
		edge from parent
		node [above] {$(a_3-a_2,\lambda_3)$}}
   ;
\end{tikzpicture}
\end{center}
\egroup

\bgroup
\item[$x_3$:]\
\begin{center}

\begin{tikzpicture}[
        thick,
        level/.style={level distance=2.5cm,sloped},
        level 2/.style={sibling distance=2.75cm},
        level 3/.style={sibling distance=2cm},
        level 4/.style={sibling distance=1.75cm},
        grow=right]
    \node {$x_3$}
    	child{node {$R_4$}
		child{node{$R_3$}
			child{node{$R_1$}
				edge from parent
				node [above] {$(a_1, \lambda_1)$}}
			edge from parent
			node[above]{$(a_3-a_2,\lambda_3)$}}
		edge from parent
		node [above] {$(-a_3,\lambda_4)$}}
		   ;
\end{tikzpicture}
\end{center}
\egroup

\end{itemize}

The possible valid paths are as follows
\[(R_1), (R_1, R_2), (R_1, R_2, R_3), (R_3), (R_3, R_1), (R_4, R_3, R_1)\]
\[(R_3, R_4, R_3), \ldots, (R_3, n\times(R_4, R_3))\]
\[(R_3, R_4, R_3, R_1), \ldots, (R_3, n\times(R_4, R_3), R_1)\]
\[(R_3,R_1, R_2,R_3), \ldots, (R_3, n\times(R_1, R_2, R_3))\]

 Since $(R_1, R_2)$ always exists, $(R_1)$ will never have the smallest slope, and since $(R_4, R_3, R_1)$ always exists, none of $(R_3, R_4, R_3, R_1), \ldots, (R_3, n\times(R_4, R_3), R_1)$ nor $(R_3, R_1)$ will have smallest slope. In addition, we only need to consider $(R_3, k \times (R_4, R_3))$ corresponding to the longest valid path in its sequence, and $(R_3, k\times(R_1,R_2,R_3))$ corresponding to the shortest valid path in its sequence.

 \begin{center}
\bgroup
\def\arraystretch{2.25}
\begin{tabular}{c | c | c}
Saddle Connection & Slope & Existence Condition(s)\\
\hline
$(R_1, R_2)$ & $\dfrac{a_2}{\lambda_1+\lambda_2}$ & None\\

$(R_1, R_2, R_3)$ & $\dfrac{a_3}{\lambda_1+\lambda_2+\lambda_3}$ & $a_2 > \dfrac{(\lambda_1 + \lambda_1)a_3}{\lambda_1+\lambda_2+\lambda_3}$\\

$(R_3)$ & $\dfrac{a_3 - a_2 }{\lambda_3}$ & $a_3 > a_2$\\

$(R_3, k \times (R_4, R_3))$ & $\dfrac{a_3 -(k+1)a_2}{(k+1)\lambda_3 + k\lambda_4}$ & \parbox{6.5cm}{\vspace{.75\baselineskip}\centering $a_3 -(k+1)a_2 > 0$,\\$(k+1)\lambda_3 + k\lambda_4 \leq 1$}\\

$(R_3, k\times(R_1,R_2,R_3))$ & $\dfrac{(k+1)a_3-a_2}{k(\lambda_1+\lambda_2) + (k+1)\lambda_3}$  & \parbox{7.5cm}{\vspace{.75\baselineskip} \centering $\dfrac{(\lambda_1+\lambda_2+\lambda_3)((k+1)a_3-a_2)}{k(\lambda_1+\lambda_2) + (k+1)\lambda_3} <a_2 < 2a_3$,\\$a_3 < a_2 + \dfrac{\lambda_3((k+1)a_3-a_2)}{k(\lambda_1+\lambda_2) + (k+1)\lambda_3}$,\\$ k(\lambda_1+\lambda_2) + (k+1)\lambda_3\leq 1$}\\

$(R_4, R_3, R_1)$ & $\dfrac{a_1-a_2}{\lambda_1+\lambda_3+\lambda_4}$ & None
\end{tabular}
\egroup
\end{center}

 First, assume $a_2 \geq a_3$, in this case $(R_1, R_2, R_3)$ always exists and will beat $(R_1, R_2)$. Thus the candidates for smallest slope are $(R_1, R_2, R_3)$, $(R_3, k\times(R_1,R_2,R_3))$ and $(R_4, R_3, R_1)$. Note $(R_3, k\times(R_4, R_3))$ is not a valid option in this case since we begin by passing to the consecutive rectangle.
 
 Now assume $a_2 < a_3$, $(R_3)$ exists in this case which implies that neither $(R_1, R_2, R_3)$ nor $(R_3, k\times(R_1,R_2,R_3))$ will have smallest slope. Thus the possible candidates are $(R_3)$, $(R_1, R_2)$, $(R_3, k \times (R_4, R_3))$, and $(R_4, R_3, R_1)$.

\subsection{$\pi = (4321)$}
\subsubsection{The Poincar\'e Section}
The equalities that describe the space:
\begin{itemize}
\item $\sum\lambda_i h_i = 1$
\item $h_1 = a_1$
\item $h_4 = a_3$
\item $h_1=h_2-a_2$
\item $h_4 = h_3 - a_2$
\end{itemize}
The inequalities that describe the section:
\begin{itemize}
\item $\sum \lambda_i \leq 1$
\item $\lambda_i > 0$
\item $0 < a_1 \leq h_2$
\item $0 < a_2 \leq h_2, h_3$
\item $0 < a_3 \leq h_3$
\end{itemize}
\subsubsection{Bounding the return times} Since $h_4=a_3$, no saddle connections will begin at $x_3$.

 \begin{itemize}
  
\bgroup 
\item[$x_0$:]\
\begin{center}

\begin{tikzpicture}[
        thick,
        level/.style={level distance=2.5cm,sloped},
        level 2/.style={sibling distance=2.25cm},
        level 3/.style={sibling distance=2cm},
        level 4/.style={sibling distance=1.75cm},
        grow=right]
    \node {$x_0$}
        child  {node {$R_1$}{
                  child { node {$R_2$}
            	child{node {$R_3$}
			  edge from parent
                    node [above] {$(a_3-a_2,\lambda_3)$}}
                    edge from parent
			node [above] {$(a_2-a_1,\lambda_2)$}
            }
            edge from parent 
            node [above] {$(a_1,\lambda_1)$}
        }}    ;
\end{tikzpicture}
\end{center}
\egroup

\bgroup
\item[$x_1$:]\
\begin{center}

\begin{tikzpicture}[
        thick,
        level/.style={level distance=2.5cm,sloped},
        level 2/.style={sibling distance=3cm},
        level 3/.style={sibling distance=2.75cm},
        level 4/.style={sibling distance=1.75cm},
        level 6/.style={level distance=1cm},
        grow=right]
    \node {$x_1$}
    	child{node {$R_2$}
		child {node {$R_1$}
			child {node {$R_2$}
				child {node {$R_1$}
					child {node {$R_2$}
						child {node {$\cdots$}}
						edge from parent
						node [above] {$(a_2-a_1, \lambda_2)$}}
					edge from parent
					node[above] {$(a_1, \lambda_1)$}}
				child {node{$R_3$}
					child{node{$R_2$}
						child{node{$\cdots$}}
						edge from parent
						node[above]{$(a_2-a_1, \lambda_2)$}}
					edge from parent
					node [above] {$(a_3 - a_2, \lambda_3)$}}
				edge from parent
				node [above] {$(a_2 - a_1, \lambda_2)$}}
			edge from parent
			node [above] {$(a_1, \lambda_2)$}}
		child{node{$R_3$}
			child{node{$R_2$}
				child{node{$\cdots$}}
				edge from parent
				node [above] {$(a_2-a_1,\lambda_2)$}}
			child{node{$R_4$}
				child {node {$R_3$}
					child{node{$\cdots$}}
					edge from parent
					node [above] {$(a_3-a_2,\lambda_3)$}}
				edge from parent
				node [above] {$(a_1, \lambda_1)$}}
			edge from parent
			node[above]{$(a_3-a_2,\lambda_3)$}}
		edge from parent
		node [above] {$(a_2-a_1,\lambda_4)$}}
		   ;
\end{tikzpicture}
\end{center}
\egroup

\bgroup
\item[$x_2$:]\
\begin{center}

\begin{tikzpicture}[
        thick,
        level/.style={level distance=2.5cm,sloped},
        level 2/.style={sibling distance=2.75cm},
        level 3/.style={sibling distance=2.75cm},
        level 4/.style={level distance=1cm},
        grow=right]
    \node {$x_2$}
    	child {node {$R_3$}
		child{node {$R_4$}
			child{node {$R_3$}
				child{node{$\cdots$}}
				edge from parent
				node[above]{$(a_3-a_2, \lambda_3)$}}
			edge from parent
			node [above] {$(-a_3,\lambda_4)$}}
		edge from parent
		node[above]{$(a_3-a_2, \lambda_3)$}}
   ; 
\end{tikzpicture}
\end{center}
\egroup


\end{itemize}

The possible valid paths are as follows
\[(R_1), (R_1, R_2), (R_1,R_2,R_3), (R_2), (R_2, R_3), (R_3)\]
\[(R_2, R_3, R_4, R_3), \ldots, (R_2, R_3, n\times(R_4, R_3))\]
\[(R_2, R_3, R_2), \ldots, (R_2, n \times(R_3, R_2))\]
\[(R_2, R_1, R_2, R_3), \ldots, (R_2, R_1, n \times (R_2, R_3))\]
\[(R_2, R_1, R_2), \ldots, (R_2, n \times(R_1, R_2))\]
\[(R_3, R_4, R_3), \ldots, (R_3, n \times(R_4, R_3))\]

Note we only have to consider $(R_2, R_3, k \times(R_4, R_3))$, $(R_2, k\times (R_3, R_2))$, $(R_2, R_1, k \times(R_2, R_3))$, and $(R_3, k \times(R_4, R_3))$ corresponding to the longest valid path in their corresponding sequences, and $(R_2, k \times(R_1, R_2))$ corresponding to the shortest valid path in its sequence.

 \begin{center}
\def\arraystretch{2.25}
\begin{tabular}{c | c | c}
Saddle Connection & Slope & Existence Condition(s)\\
\hline
$(R_1)$ & $\dfrac{a_1}{\lambda_1}$ & None\\

$(R_1, R_2)$ & $\dfrac{a_2}{\lambda_1+\lambda_2}$ & $a_1 > \dfrac{\lambda_1(a_2)}{\lambda_1+\lambda_2}$\\
$(R_1, R_2, R_3)$ & $\dfrac{a_3}{\lambda_1+\lambda_2+\lambda_3}$ & $a_1 > \dfrac{\lambda_1(a_2)}{\lambda_1+\lambda_2}$, $a_2 > \dfrac{(\lambda_1+\lambda_2)a_3}{\lambda_1+\lambda_2+\lambda_3}$ \\

$(R_2)$ & $\dfrac{a_2-a_1}{\lambda_2}$ & $a_2 > a_1$\\

$(R_2, R_3)$ & $\dfrac{a_3 -a_1}{\lambda_2 + \lambda_3}$ & $a_3 > a_1$, $a_2 > \dfrac{\lambda_2(a_3 - a_1)}{\lambda_2+\lambda_1} + a_1$\\

$(R_2, R_3, k \times(R_4, R_3))$ & $\dfrac{a_3 - ka_2 - a_1}{k(\lambda_2 +\lambda_4)+ (k+1)\lambda_3}$ & \parbox{7cm}{\vspace{.75\baselineskip}\centering $a_3 > ka_2 + a_1$,\\$a_ 2> a_1 + \dfrac{\lambda_2(a_3 - ka_2 - a_1)}{k(\lambda_2 +\lambda_4)+ (k+1)\lambda_3}$, \\ $k(\lambda_2 +\lambda_4)+ (k+1)\lambda_3 \leq 1$}\\

$(R_2, k\times (R_3, R_2))$ & $\dfrac{a_2 + ka_3 - (k+1)a_1}{(k+1)\lambda_2 + k\lambda_3}$ & \parbox{7cm}{\vspace{.75\baselineskip}\centering $(k+1)a_1 < a_2 + ka_3$,\\ $a_2 > \dfrac{\lambda_2(a_2 + ka_3 - (k+1)a_1)}{(k+1)\lambda_2 + k\lambda_3} + a_1$,\\$(k+1)\lambda_2 + k\lambda_3\leq1$}\\

$(R_2, R_1, k \times (R_2, R_3))$ & $\dfrac{a_2 + ka_3 - ka_1}{\lambda_1+(k+1)\lambda_2+k\lambda_3}$ & \parbox{8cm}{\vspace{.75\baselineskip}\centering $ka_1 < a_2 + ka_3$\\$a_2 > (a_1-a_2) + \dfrac{(\lambda_1 + 2\lambda_2)(a_2 + ka_3 - ka_1)}{\lambda_1+(k+1)\lambda_2+k\lambda_3}$\\$\lambda_1+(k+1)\lambda_2+k\lambda_3 \leq 1$} \\

$(R_2, k \times(R_1, R_2))$ & $\dfrac{ka_2 - a_1}{(k-1)\lambda_1+k\lambda_2}$ & \parbox{5cm}{\vspace{.75\baselineskip}\centering $ka_2 > a_1$,\\$(k-1)\lambda_1+k\lambda_2 \leq 1$}\\

$(R_3)$ & $\dfrac{a_3 - a_2 }{\lambda_3}$ & $a_3 > a_2$\\

$(R_3, k \times(R_4, R_3))$ & $\dfrac{a_3 -ka_2}{n\lambda_3 + (k-1)\lambda_4}$ & \parbox{5cm}{\vspace{.75\baselineskip}\centering $a_3 > ka_2$,\\$k\lambda_3 + (k-1)\lambda_4 \leq 1$}\\

\end{tabular}
\end{center}

If $a_1 < a_2$, then $(R_2)$ always exists, which implies that $(R_1, R_2)$ and $(R_2, k \times(R_1, R_2))$ will never have the smallest slope.

If $a_1 < a_2 < a_3$, $(R_3)$ always exists, which implies that $(R_2, R_3)$, $(R_1, R_2, R_3)$, $(R_2, R_1, k \times (R_2, R_3))$ will never have smallest slope. Thus the candidates are $(R_1)$, $(R_2)$, $(R_2, R_3, k \times(R_4, R_3))$, $(R_2, k\times (R_3, R_2))$, $(R_3)$, and $(R_3, k \times(R_4, R_3))$.

If $a_1 < a_3 \leq a_2$, $(R_2, R_3)$ always exists which implies $(R_1, R_2, R_3)$ and $(R_2, k\times (R_3, R_2))$ will never have the smallest slopes. Thus the candidates are $(R_1)$, $(R_2)$, $(R_2, R_3)$, $(R_2, R_3, k \times(R_4, R_3))$, and $(R_2, R_1, k \times (R_2, R_3))$.

If $a_3 \leq a_1 < a_2$, then $(R_1, R_2, R_3)$ always exists, which implies that $(R_1)$, $(R_1, R_2)$ will never have smallest slope. Thus the candidates are $(R_1, R_2, R_3)$, $(R_2)$, and $(R_2, R_1, k \times (R_2, R_3))$.

If $a_1 \geq a_2$, $(R_2)$ $(R_2, R_3)$, $(R_2, R_3, k\times(R_4, R_3))$, and $(R_2, k \times(R_3, R_2))$, do not exit. $(R_1, R_2)$ always exists which implies that $(R_1)$ will never have the smallest slope.

If $a_2 \leq a_1 < a_3$ or  $a_2 < a_3 \leq a_1$, then $(R_3)$ always exists, which implies that  $(R_1, R_2, R_3)$ and $(R_2, R_1, k \times (R_2, R_3))$ do not have the smallest slope. Thus the candidates are $(R_1, R_2)$, $(R_2,k\times(R_1, R_2))$, $(R_3)$, $(R_3,k\times(R_4, R_3))$.

If $a_3 < a_2 < a_1$, $(R_3)$ and $(R_3, k \times(R_4, R_3))$ do not exist, however, $(R_1, R_2, R_3)$ always exists which implies that $(R_1, R_2)$ will never have the smallest slope. Thus the candidates are $(R_1, R_2, R_3)$, $(R_2, R_1, k \times (R_2, R_3))$, $(R_2,k\times(R_1, R_2))$.

\subsection{$\pi = (4213)$}
\subsubsection{The Poincar\'e Section}
The equalities that describe the space:
\begin{itemize}
\item $\sum\lambda_i h_i = 1$
\item $h_3 = a_3$
\item $h_4 = a_3$
\item $h_1=h_2-a_2$
\item $h_4 = h_2 - a_1$
\end{itemize}
The inequalities that describe the section:
\begin{itemize}
\item $\sum \lambda_i \leq 1$
\item $\lambda_i > 0$
\item $0 < a_1 < h_1$
\item $0 < a_2 < a_3$
\end{itemize}
\subsubsection{Bounding the return times} Since $h_4=h_3$, no saddle connections can start at $x_3$.

\begin{itemize}
\item[$x_0$:]\
\begin{center}

\begin{tikzpicture}[
        thick,
        level/.style={level distance=2.5cm,sloped},
        level 2/.style={sibling distance=2.25cm},
        level 3/.style={sibling distance=2cm},
        level 4/.style={sibling distance=1.75cm},
        grow=right]
    \node {$x_0$}
        child  {node {$R_1$}{
                  child { node {$R_2$}
            	child{node {$R_3$}
			  edge from parent
                    node [above] {$(a_3-a_2,\lambda_3)$}}
                    edge from parent
			node [above] {$(a_2-a_1,\lambda_2)$}
            }
            edge from parent 
            node [above] {$(a_1,\lambda_1)$}
        }}    ;
\end{tikzpicture}
\end{center}

\item[$x_1$:]\
\begin{center}

\begin{tikzpicture}[
        thick,
        level/.style={level distance=2.5cm,sloped},
        level 2/.style={sibling distance=2cm},
        level 3/.style={sibling distance=2.75cm},
        level 4/.style={sibling distance=2cm},
        level 6/.style={level distance=1cm},
        grow=right]
    \node {$x_1$}
    	child {node {$R_2$}
		child{node {$R_1$}
			child{node {$R_2$}
				child {node {$R_1$}
					child {node {$R_2$}
						child {node {$\cdots$}}
						edge from parent
						node [above] {$(a_2 - a_1, \lambda_2)$}}
					edge from parent
					node [above] {$(a_1, \lambda_1)$}}
				child{node{$R_3$}
					edge from parent
					node[above]{$(a_3-a_2,\lambda_3)$}}
				edge from parent
				node[above] {$(a_2-a_1,\lambda_2)$}}
			edge from parent
			node [above] {$(a_1,\lambda_1)$}}
		child{node {$R_3$}
			child{ node{$R_4$}
				child{node{$R_2$}
					child{node{$\cdots$}}
					edge from parent
					node[above]{$(a_2-a_1,\lambda_2)$}}
				edge from parent
				node[above] {$(-a_3, \lambda_4)$}}
			edge from parent
			node [above] {$(a_3-a_2,\lambda_3)$}}
		edge from parent
		node [above] {$(a_2-a_1,\lambda_2)$}}
   ;
\end{tikzpicture}
\end{center}

\item[$x_2$:]\
\begin{center}

\begin{tikzpicture}[
        thick,
        level/.style={level distance=2.5cm,sloped},
        level 2/.style={sibling distance=2.25cm},
        level 3/.style={sibling distance=2cm},
        level 4/.style={sibling distance=1.75cm},
        grow=right]
    \node {$x_2$}
    	child{node {$R_3$}
		edge from parent
		node [above] {$(a_3-a_2,\lambda_3)$}}
		   ;
\end{tikzpicture}
\end{center}

\end{itemize}

The valid paths are as follows
\[(R_1), (R_1, R_2), (R_1, R_2, R_3), (R_2), (R_2, R_3), (R_3),\]
\[(R_2, R_3, R_4, R_2), \ldots, (R_2, n\times(R_3, R_4, R_2)),\]
\[(R_2, R_3, R_4, R_2, R_3), \ldots, (R_2, R_3, n\times(R_4, R_2, R_3)),\]
\[(R_2, R_1, R_2, R_3), \ldots, (R_2, n\times(R_1, R_2), R_3),\]
\[(R_2, R_1, R_2), \ldots, (R_2, n\times(R_1, R_2))\]

Since $(R_3)$ always exists, $(R_1, R_2, R_3)$, $(R_2, R_3)$, and $(R_2, R_1, R_2, R_3), \ldots, (R_2, n\times(R_1, R_2), R_3)$ will never have the smallest slope by Lemma \ref{comp}.2. We only need to consider $(R_2, k \times(R_3, R_4, R_2)$ the longest valid path in its sequence and $(R_2, R_3, k \times(R_4, R_2, R_3)$, $(R_2, k \times(R_1, R_2), R_3)$ and $(R_2, k \times(R_1, R_2)$ the shortest valid path in their respective sequences.

\begin{center}
\bgroup
\def\arraystretch{2.25}
\begin{tabular}{c | c | c}
Saddle Connection & Slope & Existence Condition(s)\\
\hline
$(R_1)$ & $\dfrac{a_1}{\lambda_1}$ & None\\
$(R_1, R_2)$ & $\dfrac{a_2}{\lambda_1+\lambda_2}$ & $a_1 > \dfrac{\lambda_1(a_2)}{\lambda_1+\lambda_2}$\\\
$(R_2)$ & $\dfrac{a_2-a_1}{\lambda_2}$ & $a_2 > a_1$\\
$(R_2, k \times(R_3, R_4, R_2))$ & $\dfrac{a_2 - (k+1)a_1}{(k+1)\lambda_2+k\lambda_3+k\lambda_4}$&  \parbox{8cm}{\vspace{.75\baselineskip}\centering $a_2 > (k+1)a_1$, $a_2 > a_1 + \dfrac{\lambda_1(a_2 - (k+1)a_1)}{(k+1)\lambda_2+k\lambda_3+k\lambda_4}$,\\ $a_3 > a_1 + \dfrac{(\lambda_2 + \lambda_3)(a_2 - (k+1)a_1)}{(k+1)\lambda_2+k\lambda_3+k\lambda_4}$,\\ $(k+1)\lambda_2+k\lambda_3+k\lambda_4\leq1$} \\
$(R_2, R_3, k \times(R_4, R_2, R_3))$ & $\dfrac{a_3 - (k+1)a_1}{(k+1)(\lambda_2+\lambda_3)+k\lambda_4}$ & \parbox{6cm}{\vspace{.75\baselineskip}\centering $a_3 > (k+1)a_1$,\\$a_2 > a_1 + \dfrac{\lambda_1(a_3 - (k+1)a_1)}{(k+1)(\lambda_2+\lambda_3)+k\lambda_4}$,\\$a_3 > a_1 + \dfrac{(\lambda_2 + \lambda_3)(a_2 - (k+1)a_1)}{(k+1)(\lambda_2+\lambda_3)+k\lambda_4}$,\\$(k+1)(\lambda_2+\lambda_3)+k\lambda_4 \leq 1$\\}\\
$(R_2, k \times(R_1, R_2))$ & $\dfrac{(k+1)a_2-a_1}{k\lambda_1+(k+1)\lambda_2}$ & \parbox{5cm}{\vspace{.75\baselineskip}\centering $(k+1)a_2> a_1$,\\$k\lambda_1+(k+1)\lambda_2 \leq 1$}\\
$(R_3)$ & $\dfrac{a_3-a_2}{\lambda_3}$ & None\\
\end{tabular}
\egroup
\end{center}

Assume first that $a_1 < a_2$. $(R_2)$ always exists and will have smaller slope than $(R_1, R_2)$ and $(R_2, k \times(R_1, R_2))$, by Lemma \ref{comp}.2. Thus the only possible candidates for smallest slope are $(R_1)$, $(R_2)$, $(R_2, k \times(R_3, R_4, R_2))$, $(R_2, R_3, k \times(R_4, R_2, R_3))$, and $(R_3)$.

Next assume $a_1 \geq a_2$. In this case $(R_2)$ and $(R_2, k \times(R_3, R_4, R_2))$ do not exist, however $(R_1, R_2)$ always exists, which means $(R_1)$ will never have the smallest slope, by Lemma \ref{comp}.1. If  $a_1 \geq a_3 > a_2$, $(R_2, R_3, k \times(R_4, R_2, R_3))$ does not exist and so the possible candidates for smallest slope are $(R_1)$, $(R_2, k\times(R_1, R_2)$,  and $(R_3)$. If $a_3 > a_1 > a_2$, then the possible candidates for smallest slope are $(R_1)$, $(R_2, k\times(R_1, R_2)$, $(R_2, R_3, k \times(R_4, R_2, R_3))$, and $(R_3)$.

\bibliographystyle{abbrv}
\bibliography{H2Refs}

\begin{thebibliography}{10}

\bibitem{Ath13}
J.~S. Athreya.
\newblock Gap distributions and homogeneous dynamics.
\newblock to appear in Proceedings of ICM Satellite Conference on Geometry,
  Topology, and Dyanamics in Negative Curvature.

\bibitem{AC12}
J.~S. Athreya and J.~Chaika.
\newblock The distribution of gaps for saddle connection directions.
\newblock {\em Geom. Funct. Anal.}, 22(6):1491--1516, 2012.

\bibitem{ACL}
J.~S. Athreya, J.~Chaika, and S.~Leli{\`e}vre.
\newblock The gap distribution of slopes on the golden {L}.
\newblock In {\em Recent trends in ergodic theory and dynamical systems},
  volume 631 of {\em Contemp. Math.}, pages 47--62. Amer. Math. Soc.,
  Providence, RI, 2015.

\bibitem{AC13}
J.~S. Athreya and Y.~Cheung.
\newblock A {P}oincar\'e section for the horocycle flow on the space of
  lattices.
\newblock {\em Int. Math. Res. Not. IMRN}, (10):2643--2690, 2014.

\bibitem{EM}
N.~D. Elkies and C.~T. McMullen.
\newblock Gaps in {${\sqrt n}\bmod 1$} and ergodic theory.
\newblock {\em Duke Math. J.}, 123(1):95--139, 2004.

\bibitem{Hall70}
R.~R. Hall.
\newblock A note on {F}arey series.
\newblock {\em J. London Math. Soc. (2)}, 2:139--148, 1970.

\bibitem{HS}
P.~Hubert and T.~A. Schmidt.
\newblock An introduction to {V}eech surfaces.
\newblock In {\em Handbook of dynamical systems. {V}ol. 1{B}}, pages 501--526.
  Elsevier B. V., Amsterdam, 2006.

\bibitem{Masur}
H.~Masur.
\newblock Ergodic theory of translation surfaces.
\newblock In {\em Handbook of dynamical systems. {V}ol. 1{B}}, pages 527--547.
  Elsevier B. V., Amsterdam, 2006.

\bibitem{MT}
H.~Masur and S.~Tabachnikov.
\newblock Rational billiards and flat structures.
\newblock In {\em Handbook of dynamical systems, {V}ol.\ 1{A}}, pages
  1015--1089. North-Holland, Amsterdam, 2002.

\bibitem{UW}
C.~Uyanik and G.~Work.
\newblock The distribution of gaps for saddle connections on the octagon.
\newblock {\em International Mathematics Research Notices}, 2015.

\bibitem{Vee82}
W.~A. Veech.
\newblock Gauss measures for transformations on the space of interval exchange
  maps.
\newblock {\em Ann. of Math. (2)}, 115(1):201--242, 1982.

\bibitem{Z96}
A.~Zorich.
\newblock Finite {G}auss measure on the space of interval exchange
  transformations. {L}yapunov exponents.
\newblock {\em Ann. Inst. Fourier (Grenoble)}, 46(2):325--370, 1996.

\bibitem{Zorich}
A.~Zorich.
\newblock Flat surfaces.
\newblock In {\em Frontiers in number theory, physics, and geometry. {I}},
  pages 437--583. Springer, Berlin, 2006.

\end{thebibliography}

\end{document}